\documentclass[12pt]{article}

\usepackage[english]{babel}
\usepackage{amsfonts}
\usepackage{amssymb}
\usepackage{amsthm}
\usepackage{bbold}
\usepackage{array}
\usepackage[section]{placeins}
\usepackage{float}
\usepackage{titlesec}
\usepackage{stmaryrd}
\usepackage{ulem}

\usepackage[letterpaper,top=2cm,bottom=2cm,left=3cm,right=3cm,marginparwidth=1.75cm]{geometry}

\usepackage{siunitx}
\usepackage{mathtools}
\usepackage{amsmath}
\usepackage{graphicx}
\usepackage[colorlinks=true, allcolors=blue]{hyperref}
\usepackage{caption}
\usepackage{subcaption}
\usepackage[table]{xcolor}
\usepackage{authblk}
\usepackage{csquotes}
\usepackage{tikz, tikz-3dplot}
\usepackage{pgfplots}
\usepackage{setspace}

\definecolor{lightgray}{gray}{0.9}

\newtheorem{theorem}{Theorem}
\newtheorem{remark}[theorem]{Remark}
\newtheorem{lemma}[theorem]{Lemma}
\newtheorem{proposition}[theorem]{Proposition}

\newcommand{\eps}{\varepsilon}
\newcommand{\dis}{\displaystyle}

\title{Effective approximations of solutions to highly oscillatory diffusion equations from coarse measurements}
\author{Claude Le Bris}
\author{Frédéric Legoll}
\author{Simon Ruget}
\affil{\small École Nationale des Ponts et Chaussées, Institut Polytechnique de Paris, CNRS, 6 et 8 avenue Blaise Pascal, 77455 Marne-la-Vall\'ee, France}
\affil{\small MATHERIALS project-team, Inria Paris, 48 rue Barrault, 75013 Paris, France}
\affil{\small Emails: \{claude.le-bris,frederic.legoll,simon.ruget\}@enpc.fr}

\begin{document}

\maketitle

\begin{abstract}
We approximate a diffusion equation with highly oscillatory coefficients with a diffusion equation with \textit{constant} coefficients. The approach is put in action in contexts where only partial information (namely the global energy stored in the physical system) is available. While the reconstruction of the microstructure is known to be an ill-posed problem, we show that the reconstruction of effective coefficients is possible and this even with only some coarse information. The strategy we present takes the form of a non-convex optimization problem. Homogenization theory provides elements for a rigorous foundation of the approach. Some algorithmic aspects are discussed in details. We provide a comprehensive set of numerical illustrations that demonstrate the practical interest of our strategy. The present work improves on the earlier works~\cite{le2018best,le2013approximation}.
\end{abstract}

\section{Introduction} \label{section:section1}

\paragraph{Mathematical setting.} Mathematically, the problem we address here is the following. Consider the linear diffusion equation
\begin{equation} \label{eq:diffusive}
  \left\{
  \begin{aligned}
    -\text{div}\left(A_\eps \nabla u_\eps\right) &= 0 \quad \text{in $\mathcal{D}$},
    \\
    \left(A_\eps \nabla u_\eps\right) \cdot n &= g \quad \text{on $\partial \mathcal{D}$},
  \end{aligned}
  \right.
\end{equation}
set on an open bounded domain $\mathcal{D} \subset \mathbb{R}^d$ ($d\geq 1$), and where $g$ is a boundary condition in
\begin{equation} \label{eq:def_L2m}
  L^2_m(\partial \mathcal{D}) = \left\{ g \in L^2(\partial \mathcal{D}), \ \ \int_{\partial \mathcal{D}} g = 0 \right\}.
\end{equation}
It involves a highly oscillatory coefficient $A_\eps$ taking (possibly random) values in $\mathbb{R}^{d\times d}$ and satisfying all the required assumptions so that~\eqref{eq:diffusive} is well-posed (see Section~\ref{subsection:homogenization} for a precise statement). The parameter $\eps$ denotes the characteristic length of variation of the coefficient $A_\eps$, which is assumed to be small with respect to the size of the domain~$\mathcal{D}$. In order to fix the ideas, one may consider a coefficient $A_\eps$ deriving from a $\mathbb{Z}^d$-periodic matrix field~$A_{\text{per}}$ in the sense
\begin{equation} \label{eq:periodic}
  \forall x \in \mathbb{R}^d, \ \ A_\eps(x) = A_{\text{per}}\left(\frac{x}{\eps}\right),
\end{equation}
where $A_{\text{per}}$ is a symmetric, matrix-valued field, bounded and bounded away from $0$. Our strategy is, however, not restricted to this particular instance of coefficient $A_\eps$ (at least from a computational perspective) and our assumptions will be made precise below. 

\medskip

Our purpose is to approximate the operator
\begin{equation*}
  \mathcal{L}_\eps : g \in L^2_m(\partial \mathcal{D}) \longrightarrow u_\eps(g) \in L^2_m(\mathcal{D}),
\end{equation*}
where $u_\eps(g)$ is the solution to~\eqref{eq:diffusive} in~$L^2_m(\mathcal{D})$, by an operator
\begin{equation*}
  \mathcal{L}_{\overline{A}} : g \in L^2_m(\partial \mathcal{D}) \longrightarrow u(\overline{A}, g) \in L^2_m(\mathcal{D}),
\end{equation*}
where $\overline{A}$ denotes a constant symmetric coefficient in $\mathbb{R}^{d\times d}$ to be carefully selected, and where $\overline{u} = u(\overline{A},g)$ denotes the solution (again with vanishing average on the boundary) to 
\begin{equation} \label{eq:diffusiveAbar}
  \left\{
  \begin{aligned}
    -\text{div}\left(\overline{A} \nabla \overline{u}\right) &= 0 \quad \text{in $\mathcal{D}$},
    \\
    \left(\overline{A} \nabla \overline{u}\right) \cdot n &= g \quad \text{on $\partial \mathcal{D}$}.
  \end{aligned}
  \right.
\end{equation}
Put into words, our intent is to identify the effective diffusion operator~$\mathcal{L}_{\overline{A}}$ with constant coefficients that best represents the diffusion operator~$ \mathcal{L}_\eps$ with highly oscillatory coefficients.

\medskip

In order to construct $\mathcal{L}_{\overline{A}}$, we assume that we only have access to the value of certain observables (or outputs), which are in a certain sense \textit{coarse} and are functions of the solutions to~\eqref{eq:diffusive} subjected to various boundary conditions $g$. These observables must be thought of as \textit{macroscale} responses of the system (acquired through experimental measurement for instance) to the imposed boundary conditions. In this work, we consider as our observable the stored energy $\mathcal{E}(A_\eps, g)$ defined by
\begin{equation} \label{eq:energy}
  \mathcal{E}(A_\eps, g) = \frac{1}{2} \int_{\mathcal{D}} \nabla u_\eps(g)^T A_\eps \nabla u_\eps(g) - \int_{\partial \mathcal{D}} g \, u_\eps(g) = - \frac{1}{2} \int_{\partial \mathcal{D}} g \, u_\eps(g)
\end{equation}
(the rightmost equality evidently derives from using~\eqref{eq:diffusive}). Other integrated (or coarse) quantities than the energy~\eqref{eq:energy} could be considered equally well. The point is, we use as an observable a quantity that does not explicitly require any \textit{microscale} knowledge of either $u_\eps(g)$ or $A_\eps$. In particular, we do not assume to know, for various $g\in L^2_m(\partial \mathcal{D})$, the full solution field~$u_\eps(g)$ throughout the whole computational domain $\mathcal{D}$. The only exception to that rule will of course be when we \textit{analyze} our approach and need to assess the quality of our results. On the other hand, in the \textit{operational} phase, no such information will be required. Additionally, we anticipate to have access to only a reduced amount of data -- at least from a practical perspective. We assume we have $d(d+1)/2$ boundary data~$g$ and, correspondingly, measurements of the energy.

\medskip

In comparison to our previous works~\cite{le2018best,le2013approximation}, where some more idealized settings were considered, we emphasize that the major two differences are that (a) we do not assume we have access to $u_\eps(g)$ itself throughout the domain but only to some coarse output (here, the stored energy) depending upon~$u_\eps(g)$, (b) we do not impose a distributed loading (that is a right-hand side~$f$ within~\eqref{eq:diffusive}) but some Neumann boundary conditions~$g$. This is motivated by our wish to bring our approach closer to the actual practice in mechanical engineering and materials science laboratories. It is indeed clearly too idealistic to assume the solution~$u_\eps(g)$ may be intimately known/measured/observed in its finest details. Likewise, if~\eqref{eq:diffusive} is a proxy for the elasticity equation, it is much easier to impose to a mechanical test piece a force at its surface, that is, mathematically, a Neumann boundary condition, than a force distributed throughout its bulk (that is, a right-hand side).

\paragraph{Relation to other approaches.} Our work lies at the intersection of several general fields of mathematical research.

First of all, it is an \textit{inverse problem}. Its specificity within this category, however, is that it is a \textit{multiscale} inverse problem. In the regime we consider, we implicitly assume that the size~$\eps$ of the oscillations is sufficiently small to render all the usual and efficient approaches of inverse problems unpractical, unless prohibitively powerful computing facilities are resorted to. The literature abounds in many efficient techniques to address inverse problems. Unfortunately here, our problem cannot be stated in a standard way, and $\eps$ treated as if it were of unit size. Therefore, the difficulty of a small $\eps$ must be embraced. Now, it has been known for decades that inverse problems with highly oscillatory coefficients may be intrinsically ill-posed (see the celebrated work~\cite{lions2005some}), unless they are carefully stated. Only looking at outputs of the problem, recovering the coefficient $A_\eps\in L^\infty(\mathcal{D})$ from macroscopic measurements of the solution to~\eqref{eq:diffusive} is an ill-posed problem in the limit of $\eps$ vanishing. An infinity of microstructures may indeed yield the same macroscopic outputs. Put differently, distinguishing between all the microstructures compatible with a given macroscopic behavior is only possible when strong a priori assumptions on the microstructures are made. In that direction, we mention works such as~\cite{abdulle2019numerical,frederick2014numerical,lochner2023identification} (in the same vein, see also~\cite{nolen2009fine} for some developments in a random context), which assume that the microstructure depends on a small number of parameters that are either constant throughout the domain or vary only at the macroscopic scale. Provided that the local behavior of the multi-scale medium is known, measurements of solutions are then used to identify these coarse scale parameters. Similarly in spirit, the so-called \textit{inverse homogenization} techniques (see~\cite{cherkaev2001inverse,cherkaev2008dehomogenization}) aim at determining some global properties of the microstructure based on bulk-averaged properties, again at the price of assumptions on the microstructure (e.g. two phase mixing, inclusions, \dots). Nevertheless, there are practically relevant cases where none of the above informations postulated at the microscopic scale is available, and such alternative cases are worth exploring.

Given the above ``obstruction'', we have to somehow reformulate the inverse problem. We achieve this proceeding as follows. First, we shift the focus from the identification (one way or another) of the \textit{coefficient} $A$ to that of the \textit{operator}~$\mathcal{L}$. Specifically, we build an approximation~$\mathcal{L}_{\overline{A}}$ of~$\mathcal{L}_\eps$, rather than an approximation~$\overline{A}$ of~$A_\eps$. Intuitively, we therefore circumvent the intrinsic difficulties of the Calder{\'o}n problem or any of its variants. Second, we consider problems that are expected to admit an homogenized limit with constant coefficients. Intuitively then, at least in this asymptotic regime and for such an equation, the problem is well-posed. Of course, a general arbitrary problem does not necessarily admit such an homogenized limit, but since we will be working for $\eps$ not asymptotically small but only \textit{very} small, we may succeed. Our work demonstrates this is indeed the case.

Another important remark is in order. The construction of some suitable~$\mathcal{L}_{\overline{A}}$ from~$\mathcal{L}_\eps$ can be read as a problem in \textit{operator learning}. In principle, this makes our problem a particular case of problems amenable to machine learning techniques, see~\cite{nolen2012multiscale, stuart2010inverse,bhattacharya2023learning,chung2023multi,park2022physics,halikias2023} for examples of works in this direction. The major difficulty, however, that again makes such very popular (and often astonishingly efficient) approaches unpractical here is that we only have a \textit{handful of data} available. Indeed, we typically put ourselves under the practical limitation of an experimental laboratory: the boundary data~$g$ for~\eqref{eq:diffusive} that can be implemented are only a few (we refer to e.g.~\cite{zbMATH06685276,zbMATH06566400,zbMATH02119069,ammari_uhlmann} for works on inverse problems where a similar constraint on the amount of available data is considered). There is consequently little hope to identify~$\mathcal{L}_{\overline{A}}$ using less than, say, 10 data and a machine learning technique.
 
\paragraph{Main contributions and articulation of the work.} We may summarize our contributions as follows:
\begin{enumerate}
\item We propose a strategy to define an \textit{effective} (matrix-valued) coefficient for~\eqref{eq:diffusive}. It is inspired by the work previously performed in~\cite{le2018best,le2013approximation}, which is brought here closer to the experimental practice. The approach is formulated as an optimization problem on the energy difference $\mathcal{E}(A_\eps, g) - \mathcal{E}(\overline{A}, g)$ (see~\eqref{eq:infsup} below). The strategy requires only coarse information, namely here the energy~\eqref{eq:energy} and the corresponding energy for the approximate model~\eqref{eq:diffusiveAbar}.
\item We show that our strategy is asymptotically consistent with homogenization theory in cases where the latter applies and predicts a constant homogenized coefficient: in the limit where the size~$\eps$ of the oscillations vanishes, our strategy indeed asymptotically agrees with this coefficient.
\item We provide an extensive set of numerical experiments that confirm the validity of our approach and compare its performance to some alternative strategies.
\item We introduce some noise within our setting from two different perspectives. First, we examine the robustness of our approach with respect to unexpected, extrinsic noise in the measurement. Second, we propose a variant of our approach that purposely introduces noise into the coefficient itself, in order to enhance its practical validity.
\end{enumerate}

\medskip

The findings of this work are not intended as definitive conclusions. We only offer a snapshot of our current ideas, which could motivate further research or follow-up works. These contributions may indeed be complemented in many directions. Let us only mention two such directions, both based upon the fact that, in some practical instances, some \textit{a priori} knowledge (obtained after previous experimental campaigns, for instance) may be available. Put differently, the effective coefficient~$\overline{A}$ is sometimes not sought \textit{in the blind}.

First, it is often the case in practice that a reasonably fair approximation~$\widehat{\overline{A}}$ of the effective coefficient~$\overline{A}$ (that is, of the operator~$\mathcal{L}_{\overline{A}}$) is known beforehand. In such a situation, the above approach may be adjusted upon \textit{linearizing} the problem at the neighborhood of the estimated coefficient~$\widehat{\overline{A}}$. This linearization simplifies the dependence with respect to the optimization parameter, and thus alleviates the numerical costs of the approach. A first step in that direction is presented in~\cite[Chapter~3]{PhDthesis}.

Second, the task assigned could be to not explicitly identify/approximate the effective coefficient~$\overline{A}$ (or, more precisely, again the operator~$\mathcal{L}_{\overline{A}}$), but to only \textit{classify} it as best as possible in a dictionary of a preselected set of plausible effective coefficients~$\displaystyle\{\overline{A}_j\}_{j=1,\dots,J}$ (and their associated operators~$\mathcal{L}_{\overline{A}_j}$). The above approach may again be modified in order to account for this alternative scenario, as shown by some preliminary developments in~\cite[Chapter~4]{PhDthesis}.

\paragraph{Plan.} Our article is organized as follows. In Section~\ref{section:section2}, we introduce in more details our main strategy, and we recall some results of homogenization theory. Section~\ref{section:section3big} is devoted to a comparison of our notion of effective coefficient with similar notions for coefficients obtained differently. In Section~\ref{section:section3}, we prove that, under a periodicity assumption and in the limit $\eps \to 0$, our approach recovers the coefficient provided by homogenization theory. In Section~\ref{section:section4}, we show theoretically that our strategy, which is based on the sole knowledge of coarse observables, is equivalent to a strategy similar in spirit and which exploits the knowledge of the full Neumann-to-Dirichlet map. Section~\ref{section:section5} presents the computational aspects and provides some numerical illustrations. Two test cases, both in dimension $d=2$, are examined in various regimes of $\eps$: a first one with a periodic coefficient, and a second one with a random stationary coefficient. We finally explore the effect of noise.

\section{Detailed description of our approach} \label{section:section2}

In this section, we formalize our approach in more details. Since it is inspired by homogenization theory, we start the section by recalling in Section~\ref{subsection:homogenization} its basic principles. For a complete introduction, we refer, for instance, to the textbooks~\cite{papanicolau1978asymptotic,blanc2023homogenization}. The reader familiar with the theory may easily skip Section~\ref{subsection:homogenization} and directly proceed to Section~\ref{subsection:formalization}. We introduce our approach in Section~\ref{subsection:infsupproblem}. It takes the form of an optimization problem that is described in Sections~\ref{subsection:supproblem} and~\ref{subsection:infproblem}.

\subsection{Homogenization theory as a guideline} \label{subsection:homogenization}

To recall the basic principles of homogenization theory, we focus on the periodic setting. 
Let $\mathcal{D} \subset \mathbb{R}^d$ be an open bounded domain (for any ambient dimension $d \geq 1$), and consider the equation 
\begin{equation*}
  \left\{
  \begin{aligned}
    -\text{div}\left(A_\eps \nabla u_\eps \right) &= 0 \quad \text{in $\mathcal{D}$},
    \\
    \left(A_\eps \nabla u_\eps\right)\cdot n &= g \quad \text{on $\partial \mathcal{D}$},
  \end{aligned}
  \right.
\end{equation*}
where the function $g$ is independent of $\eps$ and belongs to the space $L^2_m(\partial \mathcal{D})$ defined by~\eqref{eq:def_L2m}, i.e.
\begin{equation*}
  L^2_m(\partial \mathcal{D}) = \left\{g \in L^2(\partial \mathcal{D}), \ \ \int_{\partial \mathcal{D}} g = 0 \right\}.
\end{equation*}
We assume in this Section~\ref{subsection:homogenization} that the coefficient $A_\eps$ takes the form~\eqref{eq:periodic}, that is
\begin{equation*}
  A_\eps(x) = A_{\text{per}}\left( \frac{x}{\eps}\right),
\end{equation*}
where $A_{\text{per}}$ is $\mathbb{Z}^d$-periodic. In addition, we assume that there exist two real numbers $\alpha,\beta>0$ such that
\begin{equation} \label{eq:borne_Aper}
  A_{\text{per}} \in L^\infty(\mathbb{R}^d, \mathcal{S}_{\alpha, \beta}),
\end{equation}
with $\mathcal{S}_{\alpha, \beta}$ defined by
\begin{equation} \label{eq:Salphabeta}
  \mathcal{S}_{\alpha, \beta} = \left\{
  \begin{array}{c}
    M \in \mathbb{R}^{d\times d}_{\text{sym}} \quad \text{s.t.} \quad \forall \xi \in \mathbb{R}^d, \ \ \alpha \, |\xi|^2 \leq \xi^T M \xi
    \\
    \text{and} \quad \forall \xi, \eta \in \mathbb{R}^d, \ \ |\xi^T M \eta| \leq \beta \, |\xi| \, |\eta|
  \end{array}
  \right\},
\end{equation}
where $\mathbb{R}^{d\times d}_{\text{sym}}$ denotes the space of symmetric $d \times d$ matrices. In contrast to the periodic assumption, this boundedness assumption will hold throughout our work, in the form of~\eqref{eq:borne_Aeps} below.

The above assumption~\eqref{eq:borne_Aper} ensures the existence and the uniqueness (up to the addition of a constant) of a solution to~\eqref{eq:diffusive}. 
As a normalization, we consider $u_\eps$ the unique solution in the space $H^1_m(\mathcal{D})$ defined by
\begin{equation*}
  H^1_m(\mathcal{D}) = \left\{u \in H^1(\mathcal{D}), \ \ \int_{\partial\mathcal{D}} u = 0 \right\}.
\end{equation*}

\medskip

In the periodic setting, the solution $u_\eps$ converges weakly in $H^1(\mathcal{D})$ to $u_\star$, the unique solution in $H^1_m(\mathcal{D})$ to the equation 
\begin{equation} \label{eq:eqstar}
  \left\{
  \begin{aligned}
    -\text{div}\left(A_\star\nabla u_\star\right) &= 0 \quad \text{in $\mathcal{D}$},
    \\
    \left(A_\star \nabla u_\star\right)\cdot n &= g \quad \text{on $\partial \mathcal{D}$},
  \end{aligned}
  \right.
\end{equation}
where $A_\star \in \mathbb{R}^{d\times d}_{\text{sym}}$ is a constant and $g$-independent \textit{homogenized} coefficient whose expression is given in~\eqref{eq:astar} below. We recall that, since $A_{\text{per}}(x) \in \mathcal{S}_{\alpha, \beta}$ for any $x \in \mathbb{R}^d$, we have that $A_\star$ is a symmetric matrix which belongs to $\mathcal{S}_{\alpha, \beta}$.

In addition, the energy associated to~\eqref{eq:diffusive}, already introduced in~\eqref{eq:energy}, is given by
\begin{equation*} \label{eq:energyeps}
  \mathcal{E}(A_\eps, g) = \frac{1}{2} \int_{\mathcal{D}} \nabla u_\eps(g)^T A_\eps \nabla u_\eps(g)- \int_{\partial \mathcal{D}} g \, u_\eps(g) = -\frac{1}{2}\int_{\partial \mathcal{D}} g \, u_\eps(g),
\end{equation*}
where the latter equality is obtained using the variational formulation of~\eqref{eq:diffusive}. We will repeatedly use this rewriting in the sequel.

For any given $g \in L^2_m(\partial \mathcal{D})$, this energy converges (when $\eps \to 0$) to the energy of the homogenized system 
\begin{equation*} 
  \mathcal{E}(A_\star, g) = \frac{1}{2} \int_{\mathcal{D}} \nabla u_\star(g)^T A_\star \nabla u_\star(g) - \int_{\partial \mathcal{D}} g \, u_\star(g) = -\frac{1}{2}\int_{\partial \mathcal{D}} g \, u_\star(g).
\end{equation*}


\medskip

The homogenized coefficient in~\eqref{eq:eqstar} is given by
\begin{equation} \label{eq:astar}
  [A_\star]_{ij} = \int_Q (e_i + \nabla w_{e_i})^T A_{\text{per}} (e_j+ \nabla w_{e_j}),
\end{equation}
where $Q=(0,1)^d$ is the periodic cell of $A_{\text{per}}$, $e_i$ designates the $i$-th vector of the canonical basis of $\mathbb{R}^d$, and where, for any $p \in \mathbb{R}^d$, $w_p$ is the unique (up to the addition of a constant) solution to the \textit{corrector} equation
\begin{equation} \label{eq:corrector}
  \left\{
  \begin{aligned}
  -\text{div}\left(A_{\text{per}}(\nabla w_p + p)\right) &= 0 \quad \text{in $\mathbb{R}^d$},
  \\
  \text{$w_p$ is $Q$-periodic}.
  \end{aligned}
  \right.
\end{equation}
The formulas~\eqref{eq:astar} and~\eqref{eq:corrector} involve explicitly the coefficient $A_{\text{per}}$, which might not be available in practice. This is precisely the point of our approach to circumvent this difficulty. In any event, the piece of homogenization theory that we have just presented will be useful to theoretically investigate our approach in the setting of periodic homogenization (see Section~\ref{section:section3}).

\subsection{Our approach} \label{subsection:formalization}

We now introduce our approach. It is not restricted to the periodic setting considered in Section~\ref{subsection:homogenization}. We assume here that there exist two real numbers $\alpha,\beta>0$ such that, for any $\eps$,
\begin{equation} \label{eq:borne_Aeps}
  A_\eps \in L^\infty(\mathcal{D}, \mathcal{S}_{\alpha, \beta}),
\end{equation}
with $\mathcal{S}_{\alpha, \beta}$ defined by~\eqref{eq:Salphabeta}.


\subsubsection{An $\inf \sup$ formulation} \label{subsection:infsupproblem}

We consider a constant, symmetric, positive definite coefficient $\overline{A} \in \mathbb{R}^{d \times d}_{\text{sym}}$. For any $g\in L^2_m(\partial \mathcal{D})$, we consider the problem~\eqref{eq:diffusiveAbar} introduced above, which we reproduce here for the convenience of the reader: search for the unique solution $\overline{u}$ in $H^1_m(\mathcal{D})$ to
\begin{equation*}
  \left\{
  \begin{aligned}
    -\text{div}\left( \overline{A} \nabla \overline{u} \right) &= 0 \quad \text{in $\mathcal{D}$,}
    \\
    \left( \overline{A} \nabla \overline{u}\right)\cdot n &= g \quad \text{on $\partial \mathcal{D}$}.
  \end{aligned}
  \right.
\end{equation*}
In order to define the constant coefficient $\overline{A}_\eps \in \mathbb{R}^{d \times d}_{\text{sym}}$ 
such that the solution $\overline{u}_\eps$ to~\eqref{eq:diffusiveAbar} with $\overline{A} = \overline{A}_\eps$ best approximates the solution $u_\eps$ to~\eqref{eq:diffusive}, one may consider the following optimization problem:
\begin{equation} \label{eq:problemofkunli}
  \inf_{\overline{A} \in \mathbb{R}^{d\times d}_{\text{sym}}, \, \overline{A} > 0} \ \sup_{\scriptsize \begin{aligned} & g \in L^2_m(\partial \mathcal{D}), \\[-0.5em] \|&g\|_{L^2(\partial \mathcal{D})}=1 \end{aligned}} \| u_\eps(g) - \overline{u}(\overline{A}, g) \|_{L^2(\mathcal{D})},
\end{equation}
where $\overline{A} > 0$ indicates that we optimize upon the set of positive definite matrices. Note that the dependence on $g$ of $\overline{u}$ and $u_\eps$, solutions to respectively~\eqref{eq:diffusiveAbar} and~\eqref{eq:diffusive}, as well as the dependence on $\overline{A}$ of $\overline{u}$, are explicitly denoted. 

This corresponds \textit{mutatis mutandis} to the idea introduced in~\cite{le2013approximation} for a different problem, namely a diffusion equation with homogeneous Dirichlet boundary conditions and a right hand side imposed on the whole domain $\mathcal{D}$. This problem~\eqref{eq:problemofkunli}, although quite natural to consider, is challenging since the norm in~\eqref{eq:problemofkunli} uses information at the fine scale, which may not be easily accessible from an experimental point of view: it requires, for all (or at least some) boundary conditions $g$, to know the solution $u_\eps(g)$ throughout the domain $\mathcal{D}$. Here we circumvent this limitation and show that only having much coarser information (namely the value of the energy~\eqref{eq:energy}) may be sufficient to identify accurate effective coefficients. 

\medskip

Rather than~\eqref{eq:problemofkunli}, we consider the following optimization problem:
\begin{equation} \label{eq:infsup}
   I_\eps = \inf_{\overline{A} \in \mathcal{S}_{\alpha, \beta}} \ \sup_{\scriptsize \begin{aligned} &g \in L^2_m(\partial \mathcal{D}), \\[-0.5em] \|&g\|_{L^2(\partial \mathcal{D})}=1 \end{aligned}} \left| \mathcal{E}(A_\eps, g) - \mathcal{E}(\overline{A}, g) \right|,
\end{equation}
where $\mathcal{S}_{\alpha, \beta}$ is defined in~\eqref{eq:Salphabeta} as the set of constant, symmetric, $\alpha$-coercive and $\beta$-bounded matrices. The energy $\mathcal{E}(A_\eps, g)$ is given by~\eqref{eq:energy} while $\mathcal{E}(\overline{A}, g)$ is obviously given by
\begin{equation*}
  \mathcal{E}(\overline{A}, g) = \frac{1}{2} \int_{\mathcal{D}} \nabla \overline{u}(g)^T \, \overline{A} \, \nabla \overline{u}(g) -\int_{\partial \mathcal{D}} g \, \overline{u}(g) = -\frac{1}{2}\int_{\partial \mathcal{D}} g \, \overline{u}(g).
\end{equation*}
Note that, in~\eqref{eq:infsup}, we choose to optimize over all boundary conditions in $L^2_m(\partial \mathcal{D})$. Instead, we could have considered the problem
\begin{equation} \label{eq:infsupH1demi}
  \inf_{\overline{A} \in \mathcal{S}_{\alpha, \beta}} \ \sup_{\scriptsize \begin{aligned} &g \in H^{-1/2}_m(\partial \mathcal{D}), \\[-0.5em] \|&g\|_{H^{-1/2}(\partial \mathcal{D})}=1 \end{aligned}} \left| \mathcal{E}(A_\eps, g) - \mathcal{E}(\overline{A}, g) \right|,
\end{equation}
where $\dis H^{-1/2}_m(\partial \mathcal{D}) = \left\{ g \in H^{-1/2}(\partial \mathcal{D}), \ \ \langle g, 1 \rangle_{H^{-1/2}(\partial \mathcal{D}), H^{1/2}(\partial \mathcal{D})} = 0 \right\}$, and where $\mathcal{E}(A_\eps, g)$ is redefined as $\langle g, u_\eps(g)\rangle_{H^{-1/2}(\partial \mathcal{D}), H^{1/2}(\partial \mathcal{D})}$ and similarly for $\mathcal{E}(\overline{A},g)$. Problem~\eqref{eq:infsupH1demi} is somewhat more natural from a theoretical perspective. Since $L^2_m(\partial \mathcal{D}) \subset H^{-1/2}_m(\partial \mathcal{D})$, we of course have
\begin{equation} \label{eq:tempo}
\sup_{\scriptsize g \in H^{-1/2}_m(\partial \mathcal{D})} \frac{\left| \mathcal{E}(A_\eps, g) - \mathcal{E}(\overline{A}, g) \right|}{\|g\|^2_{H^{-1/2}(\partial \mathcal{D})}} \geq \sup_{\scriptsize g \in L^2_m(\partial \mathcal{D})} \frac{\left| \mathcal{E}(A_\eps, g) - \mathcal{E}(\overline{A}, g) \right|}{\|g\|^2_{H^{-1/2}(\partial \mathcal{D})}},
\end{equation}
and using the density of $L^2_m(\partial \mathcal{D})$ in $H^{-1/2}_m(\partial \mathcal{D})$, we see that the above bound is actually an equality. However, the right-hand side of~\eqref{eq:tempo} is different from the quantity $\dis \sup_{\scriptsize g \in L^2_m(\partial \mathcal{D})} \frac{\left| \mathcal{E}(A_\eps, g) - \mathcal{E}(\overline{A}, g) \right|}{\|g\|^2_{L^2(\partial \mathcal{D})}}$ that we consider in~\eqref{eq:infsup}.

There are two motivations for considering~\eqref{eq:infsup} rather than~\eqref{eq:infsupH1demi}. First, it is easier in practice to manipulate functions of unit $L^2$-norm. Second, we are not able to generalize to the case~\eqref{eq:infsupH1demi} the proof of Proposition~\ref{prop:Iepsto0} and, thus, to show that our approach recovers the homogenized coefficient $A_\star$ in the limit $\eps \to 0$ (see Proposition~\ref{prop:quasiminimizers}). 

\begin{remark} \label{remark:spaceA}
  In contrast to~\eqref{eq:problemofkunli}, we minimize in~\eqref{eq:infsup} over the set $\mathcal{S}_{\alpha, \beta}$ defined in~\eqref{eq:Salphabeta}. In particular, the upper bound is mandatory in order to prove our consistency result in Section~\ref{section:section3}. From a modelling point of view, this constraint is natural. In particular, in the periodic setting, we know that, if $A_{\text{per}}(x)$ belongs to $\mathcal{S}_{\alpha, \beta}$ for any $x$, then the homogenized coefficient $A_\star$ belongs to $\mathcal{S}_{\alpha, \beta}$. From a computational point of view, we omit the constraints and work with matrices in $\mathbb{R}^{d \times d}_{\text{sym}}$. This relaxation never raised any issue in our simulations in the sense that the optimization algorithm never explored outside of $\mathcal{S}_{\alpha, \beta}$. 
\end{remark}

Note that~\eqref{eq:infsup} depends on the knowledge of the observable $\mathcal{E}(A_\eps, g)$ (that could be e.g. experimentally measured) but neither on $A_\eps$, nor on the fields $u_\eps$ themselves.

In the following section, we study separately the $\sup$ and the $\inf$ problems of the formulation~\eqref{eq:infsup}. 

\begin{remark}
We focus here on the (scalar-valued) diffusion problem~\eqref{eq:diffusive}, complemented with Neumann boundary conditions. In principle, our approach~\eqref{eq:infsup} can be extended to the (vector-valued) linear elasticity setting, with again Neumann-type boundary conditions (which amount, in this context, to loading the material through traction boundary conditions). The energy $\mathcal{E}$ is then the mechanical energy stored in the system. We do not pursue in that direction.
\end{remark}

\subsubsection{The $\sup$ problem} \label{subsection:supproblem}

We focus here on the $\sup$ subproblem within~\eqref{eq:infsup} and first show that, for any fixed $\overline{A} \in \mathcal{S}_{\alpha, \beta}$, the supremum $\Psi_\eps(\overline{A})$ is attained, where $\Psi_\eps(\overline{A})$ is defined by
\begin{equation} \label{eq:supproblem}
  \Psi_\eps(\overline{A}) = \sup_{\scriptsize \begin{aligned} &g \in L^2_m(\partial \mathcal{D}), \\[-0.5em] \|&g\|_{L^2(\partial \mathcal{D})}=1 \end{aligned}} \psi_\eps(\overline{A},g),
\end{equation}
with
\begin{equation} \label{eq:smallpsi}
  \psi_\eps(\overline{A},g) = \left| \mathcal{E}(A_\eps, g) - \mathcal{E}(\overline{A}, g) \right|.
\end{equation}

Let $\overline{A} \in \mathcal{S}_{\alpha, \beta}$ be given. We introduce the operator $\mathcal{T}_{\overline{A}}$ such that, for any $g \in L^2_m(\partial \mathcal{D})$, $\mathcal{T}_{\overline{A}} \, g = \overline{u}(g) \in L^2_m(\partial \mathcal{D})$, where $\overline{u}(g) \in H^1_m(\mathcal{D})$ is the solution to~\eqref{eq:diffusiveAbar} with zero mean value on the boundary $\partial \mathcal{D}$. We also introduce the operator $\mathcal{T}_\eps$ defined as follows: for any $g \in L^2_m(\partial \mathcal{D})$, $\mathcal{T}_\eps \, g = u_\eps(g) \in L^2_m(\partial \mathcal{D})$, where $u_\eps(g) \in H^1_m(\mathcal{D})$ is the solution to~\eqref{eq:diffusive} with zero mean value on the boundary $\partial \mathcal{D}$.

Then, we rewrite the difference~\eqref{eq:smallpsi} as
\begin{equation*}
  \psi_\eps(\overline{A},g) = \left| \int_{\partial \mathcal{D}} g \, \mathcal{H}_{\eps, \overline{A}}(g) \right|
\end{equation*}
where 
\begin{equation*}
  \mathcal{H}_{\eps, \overline{A}} := \frac{1}{2}\left(\mathcal{T}_\eps -\mathcal{T}_{\overline{A}}\right).
\end{equation*}
The operator $\mathcal{T}_{\overline{A}}$, considered as an operator from $L^2_m(\partial \mathcal{D})$ to $L^2_m(\partial \mathcal{D})$, is linear, 
compact (since the injection of $H^1(\mathcal{D})$ in $L^2(\partial \mathcal{D})$ is compact) and self-adjoint. Indeed, for any $f,g \in L^2_m(\partial \mathcal{D})$, the following holds:
\begin{align*}
  \int_{\partial \mathcal{D}} f \, \overline{u}(g) &= \int_{\mathcal{D}} \nabla \overline{u}(g) \cdot (\overline{A} \, \nabla \overline{u}(f)) &\text{[using the variational formulation of~\eqref{eq:diffusiveAbar}]} \\
  &= \int_{\mathcal{D}} \nabla \overline{u}(f) \cdot (\overline{A} \, \nabla \overline{u}(g)) & \text{[by symmetry of $\overline{A}$]} \\
  &= \int_{\partial \mathcal{D}} g \, \overline{u}(f). & \text{[using the variational formulation of~\eqref{eq:diffusiveAbar}]}
\end{align*}
Using the same manipulations, we show that the linear and compact operator $\mathcal{T}_\eps$ is self-adjoint. The operator $\mathcal{H}_{\eps, \overline{A}}$ is thus a linear, self-adjoint and compact operator. The eigenvalues of $\mathcal{H}_{\eps, \overline{A}}$ hence form a sequence of real numbers $\left\{\lambda_{\eps, \overline{A}}^j\right\}_j$ that converges to $0$. We denote by $\lambda_{\eps, \overline{A}}^{\sup}$ the eigenvalue of largest absolute value, and by $g_{\eps, \overline{A}}^{\sup}$ an associated normalized eigenmode. We thus have
\begin{equation*}
  \sup_{\scriptsize \begin{aligned} &g \in L^2_m(\partial \mathcal{D}), \\[-0.5em] \|&g\|_{L^2(\partial \mathcal{D})}=1 \end{aligned}} \psi_\eps(\overline{A},g) = \left| \lambda_{\eps, \overline{A}}^{\sup} \right|,
\end{equation*}
which is attained at $g_{\eps, \overline{A}}^{\sup}$:
\begin{equation*}
  \sup_{\scriptsize \begin{aligned} &g \in L^2_m(\partial \mathcal{D}), \\[-0.5em] \|&g\|_{L^2(\partial \mathcal{D})}=1 \end{aligned}} \psi_\eps(\overline{A},g) = \psi_\eps\left(\overline{A},g_{\eps, \overline{A}}^{\sup}\right).
\end{equation*}

\medskip

From a numerical standpoint, we will of course not be able to look for this largest eigenvalue of $\mathcal{H}_{\eps, \overline{A}}$ in the whole infinite dimensional space $L^2_m(\partial \mathcal{D})$. Instead, we introduce a finite dimensional subspace of $L^2_m(\partial \mathcal{D})$ of the form
\begin{equation} \label{eq:VnP}
  V_n^P(\partial \mathcal{D}) = \left\{ g \in L^2_m(\partial \mathcal{D}) \text{ s.t. } \exists \, (c_p)_{1\leq p\leq P} \in \mathbb{R}^P, \ g = \sum_{p=1}^P c_p \, \phi_p \text{ and } \sum_{p=1}^P c_p^2 = 1 \right\},
\end{equation}
using orthonormal functions of $L^2_m(\partial\mathcal{D})$, denoted $\{ \phi_p \}_{1\leq p \leq P}$. We now address the selection of the family $\{ \phi_p \}_{1 \leq p \leq P}$ and of the dimension $P$.

To start with, let us focus on the regime of vanishing $\eps$. Homogenization theory assesses that the operator $\mathcal{T}_\eps$ (when viewed as an operator in $\mathcal{L}(L^2_m(\partial \mathcal{D}))$) converges to $\mathcal{T}_{A_\star}$. Thus, for any $\overline{A} \neq A_\star$, the operator $\mathcal{H}_{\eps, \overline{A}}$ is well approximated by $\mathcal{H}_{\star, \overline{A}}$ defined by 
\begin{equation*}
  \mathcal{H}_{\star, \overline{A}} := \frac{1}{2}\left(\mathcal{T}_{A_\star} -\mathcal{T}_{\overline{A}}\right).
\end{equation*}
Furthermore, up to a subsequence extraction, $g_{\eps, \overline{A}}^{\sup}$ converges to an eigenvector associated to the eigenvalue of $\mathcal{H}_{\star, \overline{A}}$ with maximal absolute value. In order to define our family $\{ \phi_p \}_{1\leq p \leq P}$, we first investigate the case of spherical coefficients $A_\star = a_\star \, \text{Id}$ and $\overline{A} = \overline{a}\, \text{Id}$, and then, by extension, we use the family thus constructed in more general settings. In this case, $\mathcal{H}_{\star, \overline{A}}$ takes the form $\frac{1}{2}(a_\star^{-1} - \overline{a}^{-1}) \mathcal{R}$, where $\mathcal{R}$ designates the operator that associates to any $g \in L^2_m(\partial \mathcal{D})$ the trace on $\partial \mathcal{D}$ of the solution $w=w(g) \in H^1_m(\mathcal{D})$ to
\begin{equation} \label{eq:laplacianinverse}
  \left\{
  \begin{aligned}
    -\Delta w &= 0 \quad \text{in $\mathcal{D}$},\\
    \nabla w \cdot n &= g \quad \text{on $\partial \mathcal{D}$}.
  \end{aligned}
  \right.
\end{equation}
It would thus be sufficient to consider the eigenmodes of the latter operator with largest eigenvalue. Such eigenmodes are well defined since $\mathcal{R}$ is linear, self-adjoint and compact (for the same reasons that explain that the operators $\mathcal{T}_\eps$ and $\mathcal{T}_{\overline{A}}$ are). The operator is also positive-definite since, for any $g \in L^2_m(\partial \mathcal{D})$, we compute
$$
\int_{\partial \mathcal{D}} g \, \mathcal{R}(g) = \int_{\partial \mathcal{D}} g \, w = \int_{\mathcal{D}} \nabla w \cdot \nabla w \geq 0,
$$
where the last equality stems for the variational formulation of~\eqref{eq:laplacianinverse} (in addition, the above right-hand side vanishes if and only if $g$ vanishes). We denote $\lambda_p$ the largest $p$-th eigenvalue of the operator $\mathcal{R}$, and $\phi_p$ the associated eigenmode. If multiple eigenmodes are associated to the eigenvalue $\lambda_p$, then we consider all of them, and we duplicate the eigenvalue $\lambda_p$ in our list.

Now that we have explained how we choose $\phi_p$, let us explain how we choose $P$. We first observe that it seems necessary to consider at least $\frac{d(d+1)}{2}$ boundary conditions. Indeed, in the limit $\eps\to 0$, we have $\Psi_\eps(A_\star) \to 0$ (see Section~\ref{section:section3} below). Hence, we are to determine a constant symmetric matrix in $\mathbb{R}^{d\times d}$, that is $\frac{d(d+1)}{2}$ real numbers, from the knowledge of $P$ scalar equations $\mathcal{E}(\overline{A},g) = \mathcal{E}(A_\star,g)$. The inverse problem would obviously be underdetermined if we were to consider less than $\frac{d(d+1)}{2}$ scalar equations. Our numerical tests reveal that considering $P = \frac{d(d+1)}{2}$ is sufficient (see Section~\ref{section:changingsupbymax}).

For \textit{larger} values of $\eps$, the operator $\mathcal{H}_{\eps, \overline{A}}$ may be poorly approximated by $\mathcal{H}_{\star, \overline{A}}$. Hence, we enrich the space by including $P \geq \frac{d(d+1)}{2}$ functions $\phi_p$, that are again defined as eigenvectors of the operator $\mathcal{R}$.

\subsubsection{The $\inf$ problem} \label{subsection:infproblem}

We now focus on the $\inf$ subproblem within~\eqref{eq:infsup}, that is
\begin{equation} \label{eq:infproblem}
  \inf_{\overline{A} \in \mathcal{S}_{\alpha, \beta}} \Psi_\eps(\overline{A}),
\end{equation}
where $\Psi_\eps(\overline{A})$ is defined by~\eqref{eq:supproblem}.

This problem is a priori challenging to solve for two reasons. On the one hand, the dependence with respect to $\overline{A}$ is non-convex. On the other hand, the optimization set $\mathcal{S}_{\alpha, \beta}$ is delicate to handle.

First, the application $\overline{A} \rightarrow \Psi_\eps(\overline{A})$ is not convex. Figure~\ref{fig:psinonconvex} shows a one dimensional illustration. In practice, we apply a gradient descent together with a line search strategy to solve this $\inf$ problem. In our numerical experiments, this was sufficient to capture a satisfying approximation of $\overline{A}_\eps$.

A strategy to simplify the dependence with respect to $\overline{A}$ in some specific contexts is introduced in~\cite[Chapter~3]{PhDthesis}. The starting point lies in the fact that, in practical applications, it is common to have a ``good'' initial guess for the effective coefficient. In~\cite[Chapter~3]{PhDthesis}, the unknown coefficient $\overline{A}$ is assumed to lie in a neighbourhood of a known coefficient $\overline{A}_0$. Therefore, a perturbative development of the form
\begin{equation*}
  \overline{A} \approx \overline{A}_0 + \eta \, \overline{A}_1
\end{equation*}
makes sense (with $\eta$ the typical small distance between $\overline{A}$ and $\overline{A}_0$). Using this development, the energy can be linearized and, finally, the problem can be rewritten as an optimization problem upon an objective function quadratic in the unknown $\overline{A}_1$.

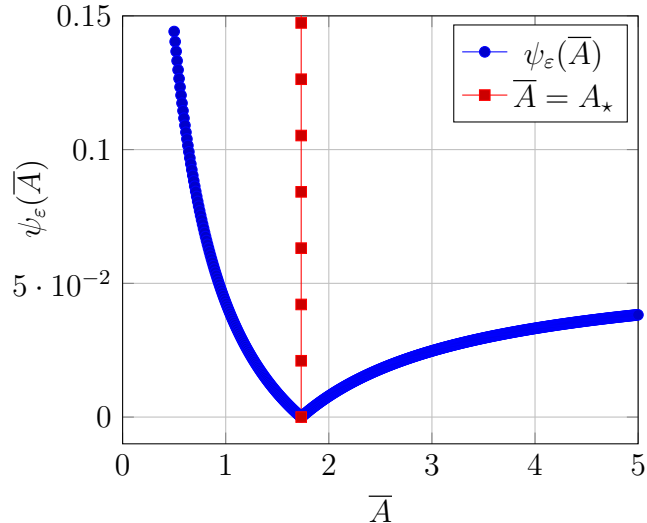
\begin{figure}[htbp]
  \centering
  \begin{tikzpicture}
    \begin{axis}[
      	xlabel={$\overline{A}$},
      	ylabel={\qquad $\psi_\eps(\overline{A})$}, 
	xmin=0, xmax=5,
	ymin=-0.01, ymax=0.15,
	grid=both,
      ]
      \addplot table[x=X, y=Y, col sep=space] {./Data/1D_Psi_Epsilon.txt};
      \addplot table[x=X, y=Y, col sep=space] {./Data/1D_Astar.txt};
      \legend{$\psi_\eps(\overline{A})$, $\overline{A} = A_\star$};
    \end{axis}
  \end{tikzpicture}
  \caption{In the one dimensional setting, we plot $\overline{A} \in (0, +\infty) \rightarrow \Psi_\eps(\overline{A})$ with $A_\eps(x) = 2 + \cos(2\pi x/\eps)$ and $\eps = 10^{-3}$.}
  \label{fig:psinonconvex}
\end{figure}

Second, and as discussed already in Remark~\ref{remark:spaceA}, the problem is solved on the whole space $\mathbb{R}^{d \times d}_{\text{sym}}$ of symmetric matrices rather than in the set $\mathcal{S}_{\alpha, \beta}$.

\section{Elements of theoretical analysis} \label{section:section3big}

In this section, we compare the effective coefficient provided by our strategy to other relevant coefficients. In Section~\ref{section:section3}, we prove that, under a periodicity assumption and in the limit $\eps \to 0$, our approach recovers the homogenized coefficient (see Proposition~\ref{prop:quasiminimizers}). In Section~\ref{section:section4}, we show theoretically (see Proposition~\ref{prop:PsiuBoundaryPsiE}) that our strategy, which is based on the sole knowledge of coarse observables, is equivalent to a strategy similar in spirit but that exploits the knowledge of the full Neumann-to-Dirichlet map (namely the strategy consisting in minimizing~\eqref{eq:supproblemboundary} below).

\subsection{Asymptotic consistency with homogenization theory} \label{section:section3}

In this section, we present some theoretical foundations for the approach introduced in Section~\ref{section:section2}. We are going to compare, in the limit when the size of the oscillations $\eps$ vanishes, our effective coefficient, obtained as a minimizer to~\eqref{eq:infsup}, with the homogenized coefficient $A_\star$ given by~\eqref{eq:astar}. For simplicity, we restrict ourselves to the periodic setting given by~\eqref{eq:periodic}. The extension of the current theoretical results to other settings is discussed in Remark~\ref{rem:extension} below. From the practical viewpoint, we recall that our approach, as described in Section~\ref{subsection:formalization}, can be put into practice in very general settings. 

We recall that~\eqref{eq:infsup} reads
$$
I_\eps = \inf_{\overline{A} \in \mathcal{S}_{\alpha, \beta}} \Psi_\eps(\overline{A}),
$$
where $\Psi_\eps$ is defined by~\eqref{eq:supproblem}.

\medskip

The proof of our consistency result follows a path similar to that in~\cite{le2018best}. The first step is to show the convergence of $I_\eps$ to $0$ as $\eps$ goes to $0$ (see Proposition~\ref{prop:Iepsto0}). Since $I_\eps$ is by definition non-negative and also smaller than $\Psi_\eps(\overline{A})$ for any admissible coefficient $\overline{A}$, this is achieved by showing the convergence of $\Psi_\eps(A_\star)$ to $0$ (recall indeed that $A_\star$ belongs to $\mathcal{S}_{\alpha, \beta}$ and is thus an admissible coefficient). This latter convergence is a direct consequence of the convergence of $A_\eps$ to $A_\star$ in the sense of homogenization.

The second step consists in showing the existence of minimizers to~\eqref{eq:infsup}, at fixed $\eps$. This is performed in Proposition~\ref{prop:exist_min} (note however that we are not able to show the uniqueness of such a minimizer).

The third step consists in proving that any subsequential limit $\overline{A}^{\#}$ to a sequence of minimizers in $\mathcal{S}_{\alpha, \beta}$ is in fact equal to $A_\star$ (see Proposition~\ref{prop:quasiminimizers}). From the convergence of $I_\eps$ to 0, we start by deducing that the bilinear form $\dis (f,g) \in L^2_m(\partial \mathcal{D}) \times L^2_m(\partial \mathcal{D}) \mapsto \int_{\partial \mathcal{D}} f \, u(\overline{A}^{\#}, g)$ is equal to the bilinear form $\dis (f,g) \in L^2_m(\partial \mathcal{D}) \times L^2_m(\partial \mathcal{D}) \mapsto \int_{\partial \mathcal{D}} f \, u(A_\star, g)$. Demonstrating that $A_\star = \overline{A}^{\#}$ then amounts to showing the invertibility of a matricial representant of the bilinear form when restricted to a selected finite-dimensional subspace of $L^2_m(\partial \mathcal{D})$ (see Lemma~\ref{lemma:lemma1}). This requires revisiting the arguments in~\cite{le2018best} since the bilinear forms considered here are different.

\medskip

We now proceed in details, and start by showing the convergence of $I_\eps$ to $0$.

\begin{proposition}[Convergence of $I_\eps$ to $0$] \label{prop:Iepsto0}
  Consider the periodic setting~\eqref{eq:periodic} with the assumption~\eqref{eq:borne_Aper}. The optimization problem~\eqref{eq:infsup} satisfies
  \begin{equation} \label{eq:Iepsto0}
    \lim_{\eps \to 0} I_\eps = 0.
  \end{equation}
\end{proposition}


\begin{proof}[Proof of Proposition~\ref{prop:Iepsto0}] 
We show that $\Psi_\eps(A_\star)$ goes to $0$, which immediately induces the limit~\eqref{eq:Iepsto0}.

We start by recalling that, for any $\overline{A} \in \mathcal{S}_{\alpha, \beta}$, the supremum $\Psi_\eps(\overline{A})$ is attained at some $g_{\eps, \overline{A}}^{\sup}$, which is a normalized eigenfunction corresponding to the eigenvalue of largest absolute value of the operator $\mathcal{H}_{\eps, \overline{A}}$ (we have shown this property in Section~\ref{subsection:supproblem}). In what follows, we are going to simply denote $g_{\eps, \overline{A}}$ this eigenfunction, and we are going to use this result for the particular choice $\overline{A} = A_\star$. We are also going to need the following a priori bounds. Let $C_{\text{Tr}}$ denote a constant such that, for any $u \in H^1(\mathcal{D})$,
\begin{equation*}
  \|u\|_{L^2(\partial \mathcal{D})} \leq C_{\text{Tr}} \, \|u\|_{H^1(\mathcal{D})},
\end{equation*}
and let $C_\mathcal{D}$ denote the Poincaré-Wirtinger constant such that, for any $u \in H^1_m(\mathcal{D})$,
\begin{equation*}
  \|u\|_{H^1(\mathcal{D})}^2 \leq C_\mathcal{D} \, \|\nabla u\|_{L^2(\mathcal{D})}^2.
\end{equation*}
The functions $u_\eps(g)$ and $u(\overline{A},g)$ can be bounded as solutions to, respectively, \eqref{eq:diffusive} and~\eqref{eq:diffusiveAbar}. Using the variational formulations of these equations, we indeed see that
\begin{equation} \label{eq:boundedueps}
  \|u_\eps(g)\|_{H^1(\mathcal{D})} \leq \frac{C_\mathcal{D} \, C_{\text{Tr}}}{\alpha} \, \|g\|_{L^2(\partial \mathcal{D})},
\end{equation}
and similarly
\begin{equation} \label{eq:boundedubar}
   \|u(\overline{A}, g)\|_{H^1(\mathcal{D})} \leq \frac{C_\mathcal{D} \, C_{\text{Tr}}}{\alpha} \, \|g\|_{L^2(\partial\mathcal{D})},
\end{equation}
where we have used that both $A_\eps$ and $\overline{A}$ are $\alpha$-coercive.

\medskip

Let us now show that $\Psi_\eps(A_\star)$ goes to $0$ when $\eps \to 0$. For any $\eps >0$, we denote $g_{\eps, A_\star}$ a maximizer of $\Psi_\eps(A_\star)$ in $L^2_m(\partial \mathcal{D})$. The sequence being bounded, it weakly converges up to an extraction, when $\eps$ tends to $0$, to an element $g_{A_\star} \in L^2(\partial\mathcal{D})$ such that 
\begin{equation} \label{eq:fcandidatebounded2}
  \|g_{A_\star}\|_{L^2(\partial \mathcal{D})} \leq \liminf_{\eps > 0} \|g_{\eps, A_\star}\|_{L^2(\partial \mathcal{D})} = 1.
\end{equation}
The weak convergence of $g_{\eps, A_\star}$ to $g_{A_\star}$ in $L^2(\partial\mathcal{D})$ directly implies that $g_{A_\star}$ belongs to $L^2_m(\partial\mathcal{D})$. We have
\begin{equation*}
  \Psi_\eps(A_\star) = \psi_\eps(A_\star, g_{\eps, A_\star}) = \Big| \mathcal{E}(A_\eps, g_{\eps, A_\star}) - \mathcal{E}(A_\star, g_{\eps, A_\star})\Big|,
\end{equation*}
and hence
\begin{multline} \label{eq:psiepsilonAstar}
  \Psi_\eps(A_\star) \leq \underbrace{\Big| \mathcal{E}(A_\eps, g_{\eps, A_\star}) - \mathcal{E}(A_\eps, g_{A_\star}) \Big|}_{(A)} + \underbrace{\Big|\mathcal{E}(A_\eps, g_{A_\star})- \mathcal{E}(A_\star, g_{A_\star})\Big|}_{(B)} \\ + \underbrace{\Big|\mathcal{E}(A_\star, g_{A_\star}) - \mathcal{E}(A_\star, g_{\eps, A_\star})\Big|}_{(C)}.
\end{multline}
We are going to successively study each of the above three terms. We start by showing the convergence of the term $(A)$. We first recall that $u_\eps(g)$ is linear with respect to $g\in L^2_m(\partial \mathcal{D})$, and that, for any $f$ and $g$ in $L^2_m(\partial \mathcal{D})$, $\dis \int_{\partial \mathcal{D}} g \, u_\eps(f) = \int_{\mathcal{D}} f \, u_\eps(g)$ (as a consequence of the symmetry of $A_\eps$). We will later on use similar properties where the field $u_\eps$ is replaced by $u_\star$. Using this, we get 
\begin{align}
  -2 \big( \mathcal{E}(A_\eps, g_{\eps, A_\star}) - \mathcal{E}(A_\eps, g_{A_\star}) \big)
  &= \int_{\partial \mathcal{D}} g_{\eps, A_\star} \, u_\eps(g_{\eps, A_\star}) - \int_{\partial \mathcal{D}} g_{A_\star} \, u_\eps(g_{A_\star})
  \nonumber \\
  &= \int_{\partial \mathcal{D}} \! \left(g_{\eps, A_\star} - g_{A_\star}\right) u_\eps(g_{\eps, A_\star}) + \int_{\partial \mathcal{D}} \! g_{A_\star} \left( u_\eps(g_{\eps, A_\star}) - u_\eps(g_{A_\star}) \right)
  \nonumber \\
  &= \int_{\partial \mathcal{D}} \left(g_{\eps, A_\star} - g_{A_\star}\right) u_\eps(g_{\eps, A_\star}) + \int_{\partial \mathcal{D}} \left( g_{\eps, A_\star}-g_{A_\star} \right) u_\eps(g_{A_\star})
  \nonumber \\
  &= \int_{\partial \mathcal{D}} \left(g_{\eps, A_\star} - g_{A_\star}\right) u_\eps(g_{\eps, A_\star} + g_{A_\star}).
  \label{eq:diffenergyepsstar}
\end{align}
The latter integrand is the product of a sequence weakly convergent to $0$ in $L^2({\partial \mathcal{D}})$ and the sequence $\dis \left(u_\eps(g_{\eps, A_\star} + g_{A_\star}) \right)_{\eps>0}$ which is strongly convergent in $L^2(\partial\mathcal{D})$. Indeed, it is bounded in $H^1(\mathcal{D})$ (as a consequence of~\eqref{eq:boundedueps} and~\eqref{eq:fcandidatebounded2}), and thus strongly convergent in $L^2(\partial \mathcal{D})$ (since the injection $H^1(\mathcal{D})$ in $L^2(\partial \mathcal{D})$ is compact). The right-hand side of~\eqref{eq:diffenergyepsstar}, and hence the term $(A)$, thus converges to $0$.

For the term $(B)$, we recall that homogenization theory implies (under the periodicity assumption~\eqref{eq:periodic}) that, for any $g\in L^2_m(\partial \mathcal{D})$, the solution $u_\eps(g)$ to~\eqref{eq:diffusive} converges to the solution $u_\star(g)$ to~\eqref{eq:eqstar}, weakly in $H^1(\mathcal{D})$. Thus, $u_\eps(g) \in L^2(\partial \mathcal{D})$ converges to $u_\star(g) \in L^2(\partial \mathcal{D})$ strongly in $L^2(\partial \mathcal{D})$, and the term $(B)$ therefore goes to $0$ as $\eps$ goes to $0$.

We eventually turn to the term $(C)$, for which we write, in a similar manner as that for the term $(A)$ (see~\eqref{eq:diffenergyepsstar}), 
\begin{equation} \label{eq:diffenergyepsstar2}
  -2 \big( \mathcal{E}(A_\star, g_{A_\star}) - \mathcal{E}(A_\star, g_{\eps, A_\star}) \big) = \int_{\partial \mathcal{D}} \left(g_{A_\star}-g_{\eps, A_\star}\right) u_\star(g_{A_\star}+g_{\eps, A_\star}).
\end{equation}
As in~\eqref{eq:diffenergyepsstar}, the latter integrand is the product of a sequence weakly convergent to $0$ in $L^2({\partial \mathcal{D}})$ with the sequence $(u_\star(g_{A_\star}+g_{\eps, A_\star}))_{\eps>0}$ which is strongly convergent in $L^2(\partial\mathcal{D})$, since it is bounded in $H^1(\mathcal{D})$ (as a consequence of~\eqref{eq:boundedubar} and~\eqref{eq:fcandidatebounded2}). The right-hand side of~\eqref{eq:diffenergyepsstar2}, and hence the term $(C)$, thus converges to $0$.

We thus have shown that the three terms in the right hand side of~\eqref{eq:psiepsilonAstar} tend to $0$. This directly implies that $\Psi_\eps(A_\star)$, and therefore $I_\eps$, converges to $0$ when $\eps$ tends to $0$. This concludes the proof of Proposition~\ref{prop:Iepsto0}.
\end{proof}

We now show the existence of a (not necessarily unique) minimizer to~\eqref{eq:infsup}.

\begin{proposition}[Existence of minimizers] \label{prop:exist_min}
  Consider the periodic setting~\eqref{eq:periodic} with the assumption~\eqref{eq:borne_Aper}. There exists a minimizer to~\eqref{eq:infsup}, which we denote $\overline{A}_\eps^{\#} \in \mathcal{S}_{\alpha, \beta}$:
  \begin{equation} \label{eq:quasiminimizerdef}
    I_\eps = \inf_{\overline{A} \in \mathcal{S}_{\alpha, \beta}} \Psi_\eps(\overline{A}) = \Psi_\eps(\overline{A}_\eps^{\#}). 
  \end{equation}
\end{proposition}

\begin{proof}[Proof of Proposition~\ref{prop:exist_min}]
  We recall (see~\eqref{eq:supproblem}) that $\Psi_\eps(\overline{A})$ is defined as the supremum over the functions $g \in L^2_m(\partial \mathcal{D})$ of unit norm of the quantity $\psi_\eps(\overline{A},g)$. It is straighforward to see that, for any $g \in L^2_m(\partial \mathcal{D})$, the function $\overline{A} \in \mathcal{S}_{\alpha, \beta} \mapsto \psi_\eps(\overline{A},g)$ is continuous. Consequently, the function $\overline{A} \mapsto \Psi_\eps(\overline{A})$ is lower semi-continuous. Since $\mathcal{S}_{\alpha, \beta}$ is compact, there exists a minimizer to $\Psi_\eps$ in $\mathcal{S}_{\alpha, \beta}$.
\end{proof}

We now state our main consistency result.

\begin{proposition}[Convergence of minimizers] \label{prop:quasiminimizers}
  Consider the periodic setting~\eqref{eq:periodic} with the assumption~\eqref{eq:borne_Aper}. Any sequence of minimizers $(\overline{A}_\eps^{\#})_{\eps>0}$ of~\eqref{eq:infsup} converges to the homogenized coefficient $A_\star$:
  \begin{equation*} 
    \lim_{\eps \to 0} \overline{A}_\eps^{\#} = A_\star.
  \end{equation*}
\end{proposition}

\begin{remark} \label{rem:extension}
Note that the periodic assumption~\eqref{eq:periodic} is not necessary to prove Proposition~\ref{prop:quasiminimizers}. We may replace it by the assumption that the sequence of matrices $\{ A_\eps \}_{\eps > 0}$ converges, in the sense of homogenization, to a \textit{constant} and \textit{symmetric} homogenized matrix $A_\star$. This is the reason why we numerically investigate, in addition to periodic test cases, a random stationary setting (see Section~\ref{sec:random_case}), which also leads to a \textit{constant} and \textit{symmetric} homogenized matrix $A_\star$.

More general cases could be considered. The approach could for instance be extended to settings where the homogenized matrix is not constant, but is rather a slowly-varying function of $x \in \mathcal{D}$. 
\end{remark}

In order to prove Proposition~\ref{prop:quasiminimizers}, we need the following lemma.

\begin{lemma} \label{lemma:lemma1}
  There exist $\frac{d(d+1)}{2}$ functions $g_{\star, k} \in L^2_m(\partial \mathcal{D})$ and $\frac{d(d+1)}{2}$ functions $\varphi_{\star, k}$ defined on $\mathcal{D}$ and affine on $\mathcal{D}$ such that the matrix $Z_\star \in \mathbb{R}^{\frac{d(d+1)}{2} \times \frac{d(d+1)}{2}}$ defined, for any $1 \leq k \leq \frac{d(d+1)}{2}$ and any $1 \leq i \leq j \leq d$, by
  \begin{equation} \label{eq:matrixZstar}
    \left\{
    \begin{aligned}
      [Z_\star]_{k,(i,i)} &= \int_{\partial \mathcal{D}} u_\star(g_{\star,k}) \, \partial_i \varphi_{\star,k} \, n_i,\\
      [Z_\star]_{k,(i,j)} &= \int_{\partial \mathcal{D}} u_\star(g_{\star,k}) \, \big(\partial_i \varphi_{\star,k} \, n_j + \partial_j \varphi_{\star,k} \, n_i\big) \quad \text{[if $j > i$]},
    \end{aligned}
    \right.
  \end{equation}
is invertible.
\end{lemma}

\begin{proof}[Proof of Lemma~\ref{lemma:lemma1}.]
We start by fixing some notation. We denote $\mathcal{C}_{\text{aff}}(\overline{\mathcal{D}})$ the set of functions defined on $\overline{\mathcal{D}}$ and affine on $\overline{\mathcal{D}}$. A symmetric matrix $A \in \mathbb{R}^{d \times d}_{\text{sym}}$ can be identified to the vector $\underline{A} \in \mathbb{R}^{\frac{d(d+1)}{2}}$ defined by $\underline{A}_{(i,j)} = A_{ij}$ for all $1\leq i\leq j \leq d$. For any $g\in L^2_m(\partial \mathcal{D})$ and any function $\varphi\in \mathcal{C}_{\text{aff}}(\overline{\mathcal{D}})$, we denote by $\underline{z}(g,\varphi)$ the vector defined by 
\begin{equation*}
  \forall 1\leq i\leq j \leq d, \ \left\{
  \begin{aligned}
    [\underline{z}(g,\varphi)]_{(i,i)} &= \int_{\partial \mathcal{D}} u_\star(g) \, \partial_i \varphi \, n_i,\\
    [\underline{z}(g,\varphi)]_{(i,j)} &= \int_{\partial \mathcal{D}} u_\star(g) \, \big(\partial_i \varphi \, n_j + \partial_j \varphi \, n_i\big) \quad \text{[if $j > i$]}.
  \end{aligned}
  \right.
\end{equation*}

\medskip

For a matrix $Z \in \mathbb{R}^{\frac{d(d+1)}{2} \times \frac{d(d+1)}{2}}$ to be invertible, it is sufficient that its rows are linearly independent as vectors in $\mathbb{R}^{\frac{d(d+1)}{2}}$. Thus, we construct iteratively the rows of the desired matrix $Z_\star$ defined by~\eqref{eq:matrixZstar}. To do so, at each step $k$, we identify a function $g_{\star, k+1} \in L^2_m(\partial \mathcal{D})$ and a function $\varphi_{\star, k+1} \in \mathcal{C}_{\text{aff}}(\overline{\mathcal{D}})$ such that the vector $\underline{z_{k+1}} = \underline{z}(g_{\star, k+1}, \varphi_{\star, k+1})$ is linearly independent of the vectors $(\underline{z_j})_{1\leq j\leq k}$ constructed previously.

\medskip

\noindent
\textbf{Initialization}. We start by identifying $(g_{\star, 1}, \varphi_{\star, 1}) \in L^2_m(\partial\mathcal{D}) \times \mathcal{C}_{\text{aff}}(\overline{\mathcal{D}})$ such that $\underline{z_1} = \underline{z}(g_{\star, 1}, \varphi_{\star, 1})$ does not vanish. We select $\varphi_{\star, 1} \in \mathcal{C}_{\text{aff}}(\overline{\mathcal{D}}) \cap L^2_m(\partial\mathcal{D})$ such that $\|\varphi_{\star, 1}\|_{L^2(\partial \mathcal{D})} \neq 0$ and set $g_{\star, 1} = \varphi_{\star, 1}$. Then, denoting $\underline{A_\star} = (A_{\star,ij})_{1\leq i\leq j \leq d} \in \mathbb{R}^{\frac{d(d+1)}{2}}$, we see that
\begin{align*}
  \underline{A_\star} \cdot \underline{z_1} &= \sum_{1\leq i,j\leq d} A_{\star,ij} \int_{\partial \mathcal{D}} u_\star(g_{\star, 1}) \, \partial_i \varphi_{\star, 1} \, n_j\\
  &= \int_{\partial \mathcal{D}} u_\star(g_{\star, 1}) \left(\sum_{1\leq i,j\leq d} A_{\star,ij} \, \partial_i \varphi_{\star, 1} \, n_j\right)\\
  &= \int_{\partial \mathcal{D}} u_\star(g_{\star, 1}) \left(A_\star \nabla \varphi_{\star, 1}\right) \cdot n\\
  &= \int_{\mathcal{D}} \nabla u_\star(g_{\star, 1}) \cdot \left(A_\star \nabla \varphi_{\star, 1}\right) + \underbrace{\int_{\mathcal{D}} u_\star(g_{\star,1}) \, \text{div}\left(A_\star \nabla \varphi_{\star,1}\right)}_{=0}\\
  &= \int_{\partial \mathcal{D}} g_{\star, 1} \, \varphi_{\star, 1}\\
  &= \int_{\partial \mathcal{D}} g_{\star, 1}^2 > 0,
\end{align*}
where we have used that $A_\star$ is symmetric in the first and fifth lines, an integration by part in the fourth line (the second integral on $\mathcal{D}$ vanishes since $\varphi_{\star, 1}$ is affine and $A_\star$ is constant) and the variational formulation of~\eqref{eq:eqstar} in the fifth line. We therefore obviously deduce that $\underline{z_1} \neq 0$.

\medskip

\noindent
\textbf{Iteration $k$}. At step $k < \frac{d(d+1)}{2}$, we assume to have constructed $k$ functions $g_{\star, 1}, \dots, g_{\star, k} \in L^2_m(\partial \mathcal{D})$ and $k$ functions $\varphi_{\star, 1}, \dots, \varphi_{\star, k} \in \mathcal{C}_{\text{aff}}(\overline{\mathcal{D}})$ such that the associated vectors $(\underline{z_j} = \underline{z}(g_{\star, j}, \varphi_{\star, j}))_{1\leq j\leq k}$ are linearly independent. We seek a function $g_{\star, k+1}\in L^2_m(\partial \mathcal{D})$ and a function $\varphi_{\star, k+1} \in \mathcal{C}_{\text{aff}}(\overline{\mathcal{D}})$ such that $\underline{z_{k+1}}$ is linearly independent of the $(\underline{z_j})_{1\leq j\leq k}$.

By contradiction, let us assume that, for every $(g, \varphi)\in L^2_m(\partial\mathcal{D})\times \mathcal{C}_{\text{aff}}(\overline{\mathcal{D}})$, there exists $(\lambda_1, \dots, \lambda_k)\in \mathbb{R}^k$ such that
\begin{equation*}
  \underline{z}(g,\varphi) = \sum_{1\leq j \leq k} \lambda_j \, \underline{z_j}.
\end{equation*}
Then, for every $\underline{A} \in \mathbb{R}^{\frac{d(d+1)}{2}}$, we have
\begin{equation*}
  \underline{A} \cdot \underline{z}(g,\varphi) = \sum_{1\leq j \leq k} \lambda_j \, \underline{A} \cdot \underline{z_j}.
\end{equation*}
Since $k<\frac{d(d+1)}{2}$, we can find an element $\underline{A_0} \in \mathbb{R}^{\frac{d(d+1)}{2}}$ such that $\underline{A_0}\neq 0$ and $\underline{A_0} \cdot \underline{z_j} = 0$ for any $1 \leq j \leq k$. This element $\underline{A_0}$ is independent of $(g, \varphi)$ and is such that
\begin{equation*}
  0 = \underline{A_0} \cdot \underline{z}(g, \varphi) = \int_{\mathcal{D}} \nabla u_\star(g) \cdot \left(A_0 \nabla \varphi \right) = \int_{\mathcal{D}} \nabla \varphi \cdot \left(A_0 \nabla u_\star(g) \right),
\end{equation*}
where $A_0 \in \mathbb{R}^{d\times d}$ denotes the symmetric matrix defined by $A_{0,ij} = (\underline{A_0})_{(i,j)}$ for all $1 \leq i \leq j \leq d$.

Introducing the matrix $P$ such that $P = A_0 \, A_\star^{-1}$, we thus get 
\begin{equation*}
  0 = \int_{\mathcal{D}} \left(P^T\nabla \varphi\right) \cdot\left(A_\star \nabla u_\star(g)\right).
\end{equation*}
Since $\text{curl}\left(P^T\nabla \varphi\right) = 0$ (recall that $\nabla \varphi$ is constant, since $\varphi$ is affine), we can find a function $\psi_\varphi$ defined on $\mathcal{D}$ such that $\nabla \psi_\varphi = P^T\nabla \varphi$. Up to replacing $\psi_\varphi$ by $\dis \psi_\varphi -\frac{1}{|\partial \mathcal{D}|}\int_{\partial \mathcal{D}}\psi_\varphi$, we can assume that $\psi_\varphi$ belongs to $L^2_m(\partial \mathcal{D})$. We thus obtain, for all $(g, \varphi) \in L^2_m(\partial \mathcal{D}) \times \mathcal{C}_{\text{aff}}(\overline{\mathcal{D}})$,
\begin{equation*}
  0 = \int_{\mathcal{D}} \nabla \psi_\varphi \cdot \left(A_\star \nabla u_\star(g)\right) = \int_{\partial \mathcal{D}} g \, \psi_\varphi.
\end{equation*}
Since we can choose $\varphi$ such that $\|\psi_\varphi\|_{L^2(\partial \mathcal{D})} \neq 0$ and take $g = \psi_\varphi$, we reach a contradiction. There hence exist $g_{\star, k+1} \in L^2_m(\partial \mathcal{D})$ and $\varphi_{\star, k+1}\in \mathcal{C}_{\text{aff}}(\overline{\mathcal{D}})$ such that $\underline{z_{k+1}} = \underline{z}(g_{\star, k+1}, \varphi_{\star, k+1})$ is linearly independent of the vectors $\underline{z_j}$ previously constructed at the steps $1\leq j\leq k$. This concludes the proof of Lemma~\ref{lemma:lemma1}. 
\end{proof}

\medskip

\begin{proof}[Proof of Proposition~\ref{prop:quasiminimizers}.]
Let $(\overline{A}_\eps^{\#})_{\eps>0}$ be a sequence of minimizers. Each coefficient is obviously bounded by $\beta$. Hence, up to an extraction that we omit to explicitly write, the sequence $(\overline{A}_\eps^{\#})_{\eps>0}$ converges to a limit that we denote $\overline{A}^{\#}$. Our aim is to prove that $\overline{A}^{\#}$ is in fact equal to $A_\star$. This is done by showing that, for any boundary condition $g \in L^2_m(\partial \mathcal{D})$, the solution $u_\star(g) \in H^1_m(\mathcal{D})$ to~\eqref{eq:eqstar} and the solution $u(\overline{A}^{\#},g) \in H^1_m(\mathcal{D})$ to~\eqref{eq:diffusiveAbar} with coefficient $\overline{A} = \overline{A}^{\#}$ are equal. We are then able to conclude using Lemma~\ref{lemma:lemma1}.

\medskip

We first show that the energies associated to the homogenized problem~\eqref{eq:eqstar} and the coarse problem~\eqref{eq:diffusiveAbar} with $\overline{A} = \overline{A}^{\#}$ are equal. On the one hand, the convergence of $\overline{A}^\#_\eps$ to $\overline{A}^{\#}$ implies the convergence in $H^1(\mathcal{D})$ of $u(\overline{A}^\#_\eps, g)$ to $u(\overline{A}^{\#},g)$, for any $g \in L^2_m(\partial \mathcal{D})$. Using the trace formula, it is then straightforward to see that $u(\overline{A}^\#_\eps,g)$ converges strongly in $L^2(\partial \mathcal{D})$ to $u(\overline{A}^{\#},g)$, and hence, for any $g \in L^2_m(\partial \mathcal{D})$,
\begin{equation} \label{eq:convergenceL2bordudiese}
  \lim_{\eps \to 0} \int_{\partial \mathcal{D}} g \, u(\overline{A}^\#_\eps,g) = \int_{\partial \mathcal{D}} g \, u(\overline{A}^{\#},g).
\end{equation}
On the other hand, Proposition~\ref{prop:Iepsto0} combined with~\eqref{eq:quasiminimizerdef} implies the convergence of $\Psi_\eps(\overline{A}_\eps^\#)$ to $0$ when $\eps$ goes to 0. Hence, for any $g \in L^2_m(\partial \mathcal{D})$, we get
\begin{equation*}
  \lim_{\eps \to 0} \Big| \int_{\partial \mathcal{D}} g \, u(\overline{A}^\#_\eps, g) - \int_{\partial \mathcal{D}} g \, u_\eps(g) \Big| = 0.
\end{equation*}
The first term above has a limit, in view of~\eqref{eq:convergenceL2bordudiese}. We recall (see Section~\ref{subsection:homogenization}) that the second term converges to $\dis \int_{\partial \mathcal{D}} g \, u_\star(g)$. We thus deduce, for any $g \in L^2_m(\partial \mathcal{D})$,
\begin{equation} \label{eq:equalityfg}
  \int_{\partial \mathcal{D}} g \, u_\star(g) = \int_{\partial \mathcal{D}} g \, u(\overline{A}^{\#},g).
\end{equation}

We now use the following polarization relation: for any $f,g \in L^2_m(\partial \mathcal{D})$,
\begin{equation*}
  \int_{\partial \mathcal{D}} (f+g) \, u_\star(f+g) = \int_{\partial \mathcal{D}} f \, u_\star(f) + \int_{\partial \mathcal{D}} g \, u_\star(g) + 2 \int_{\partial \mathcal{D}} f \, u_\star(g),
\end{equation*}
which stems from the fact that $\dis \int_{\partial \mathcal{D}} f \, u_\star(g) = \int_{\partial \mathcal{D}} g \, u_\star(f)$, and likewise for $u(\overline{A}^{\#},\cdot)$. We thus deduce from~\eqref{eq:equalityfg} that, for any $f,g \in L^2_m(\partial \mathcal{D})$,
\begin{equation*}
  \int_{\partial \mathcal{D}} f \, u_\star(g) = \int_{\partial \mathcal{D}} f \, u(\overline{A}^{\#},g),
\end{equation*}
which implies (since $u_\star(g)$ and $u(\overline{A}^{\#},g)$ belong to $L^2_m(\partial \mathcal{D})$) that, for all $g \in L^2_m(\partial \mathcal{D})$, 
\begin{equation} \label{eq:ustarequaludiese}
  u_\star(g) = u(\overline{A}^{\#},g) \quad \text{on $\partial \mathcal{D}$}.
\end{equation}

\medskip

To conclude that $\overline{A}^{\#} = A_\star$, we use Lemma~\ref{lemma:lemma1} and the matrix $Z_\star$ identified there. We introduce the vectors $\underline{A_\star}$ and $\underline{\overline{A}^\#}$ in $\mathbb{R}^{\frac{d(d+1)}{2}}$ such that $[\underline{A_\star}]_{(i,j)} = A_{\star, ij}$ for $1\leq i \leq j \leq d$, and similarly for $\underline{\overline{A}^{\#}}$. It is straightforward to see that, for any $1\leq k \leq \frac{d(d+1)}{2}$,
\begin{multline} \label{eq:long_calcul}
  [Z_\star \, \underline{A_\star}]_k
  =
  \sum_{1\leq i,j\leq d} A_{\star, ij} \int_{\partial \mathcal{D}} u_\star(g_{\star,k}) \, \partial_i \varphi_{\star, k} \, n_j
  =
  \int_{\partial \mathcal{D}} u_\star(g_{\star,k}) \, (A_\star \nabla \varphi_{\star, k}) \cdot n
  \\ =
  \int_{\mathcal{D}} \nabla u_\star(g_{\star,k}) \cdot A_\star \nabla \varphi_{\star, k}
  =
  \int_{\partial\mathcal{D}} g_{\star,k} \, \varphi_{\star, k},
\end{multline}
and
\begin{multline} \label{eq:long_calcul2}
  [Z_\star \, \underline{\overline{A}^\#}]_k = \sum_{1\leq i,j\leq d} \overline{A}^{\#}_{ij} \int_{\partial \mathcal{D}} u_\star(g_{\star,k}) \, \partial_i \varphi_{\star, k} \, n_j \\ = \sum_{1\leq i,j\leq d} \overline{A}^{\#}_{ij} \int_{\partial \mathcal{D}} u(\overline{A}^{\#}, g_{\star,k}) \, \partial_i \varphi_{\star, k} \, n_j = \int_{\partial\mathcal{D}} g_{\star,k} \, \varphi_{\star, k},
\end{multline}
where we have used~\eqref{eq:ustarequaludiese} in the second equality and a computation similar to~\eqref{eq:long_calcul} for the third equality. Collecting~\eqref{eq:long_calcul} and~\eqref{eq:long_calcul2}, we deduce
\begin{equation*}
  Z_\star \, (\underline{A_\star} - \underline{\overline{A}^\#}) = 0.
\end{equation*}
By construction (see Lemma~\ref{lemma:lemma1}), the matrix $Z_\star$ is invertible. We hence infer that $\underline{\overline{A}^\#} = \underline{A_\star}$, which means that $\overline{A}^{\#}=A_\star$.

At the beginning of the proof, we have extracted a subsequence of the minimizer sequence $(\overline{A}_\eps^{\#})_{\eps>0}$. We have just shown that such a subsequence converges to $A_\star$, a limit which is independent of the subsequence. We hence deduce that the {\it whole} sequence $\overline{A}_\eps^\#$ converges to $A_\star$. This concludes the proof of Proposition~\ref{prop:quasiminimizers}. 
\end{proof}

\subsection{Equivalence of the criterion in energy with the one using the full boundary condition} \label{section:section4}

In this section, we compare the strategy~\eqref{eq:infsup} with another strategy (defined in~\eqref{eq:supproblemboundary} below), based upon the full Neumann to Dirichlet map (that is, the complete knowledge of the trace of the function on $\partial \mathcal{D}$).

Both strategies consider an optimization problem of the type 
\begin{equation} \label{eq:strategyFieldBoundary}
  \inf_{\overline{A} \in \mathcal{S}_{\alpha, \beta}} 
  \
  \Psi_\eps^{\text{M}}(\overline{A}),
\end{equation}
where $\Psi_\eps^{\text{M}}$ is a functional depending on the available \textit{Measurement}. The strategy~\eqref{eq:infsup} exploits the functional~\eqref{eq:supproblem}. For the sake of clarity, it is temporarily renamed $\Psi_\eps^{\text{ME}}$ and we recall that
\begin{equation*}
  \Psi_\eps^{\text{ME}}(\overline{A}) = \sup_{\scriptsize \begin{aligned} &g \in L^2_m(\partial\mathcal{D}), \\[-0.5em] \|&g\|_{L^2(\partial\mathcal{D})}=1 \end{aligned}} \left| \mathcal{E}(A_\eps,g) - \mathcal{E}(\overline{A},g) \right|,
\end{equation*}
where the notation \textit{ME} stands for ``Measurement in Energy''. The strategy~\eqref{eq:supproblemboundary} is defined through the functional $\Psi_\eps^{\text{MS}}$ defined by
\begin{equation} \label{eq:supproblemboundary}
  \Psi_\eps^{\text{MS}} (\overline{A}) = \sup_{\scriptsize \begin{aligned} &g \in L^2_m(\partial\mathcal{D}), \\[-0.5em] \|&g\|_{L^2(\partial\mathcal{D})}=1 \end{aligned}} \| u_\eps(g) - u(\overline{A},g) \|_{L^2(\partial \mathcal{D})}.
\end{equation}
It relies on measurements on the boundary of $\mathcal{D}$, and \textit{MS} stands for ``Measurement on Surface'' (in contrast to the approach~\eqref{eq:problemofkunli}, which uses measurements in the interior of $\mathcal{D}$, and that we will refer to in Section~\ref{subsection:numericalresults} as the strategy \textit{MV} for ``Measurement in Volume'').

We show that the two functionals~\eqref{eq:supproblem} and~\eqref{eq:supproblemboundary} are identical, in the sense of Proposition~\ref{prop:PsiuBoundaryPsiE} below.

\begin{proposition}[Comparing $\Psi_\eps^{\text{MS}}$ and $\Psi_\eps^{\text{ME}}$] \label{prop:PsiuBoundaryPsiE}
  For any $\eps >0$ and any $\overline{A}\in \mathcal{S}_{\alpha, \beta}$, the following equality holds:
\begin{equation*}
  \Psi_\eps^{\text{ME}}(\overline{A}) = \frac{1}{2}\Psi_\eps^{\text{MS}}(\overline{A}).
\end{equation*}
\end{proposition}

This result is of interest since it assesses that, from a theoretical viewpoint, the results of the two approaches are identical, although the former approach only requires coarse measurements.
Consider indeed a sequence of minimizers $(\overline{A}_\eps^{\text{MS}})_{\eps>0}$ for~\eqref{eq:strategyFieldBoundary} with $\Psi_\eps^{\text{M}} = \Psi_\eps^{\text{MS}}$. Proposition~\ref{prop:PsiuBoundaryPsiE} induces that it is also a sequence of minimizers for~\eqref{eq:infsup}. Conversely, any sequence $(\overline{A}_\eps^{\text{ME}})_{\eps>0}$ of minimizers for~\eqref{eq:infsup} is also a sequence of minimizers for~\eqref{eq:strategyFieldBoundary} with $\Psi_\eps^{\text{M}} = \Psi_\eps^{\text{MS}}$. 

From a practical viewpoint, the two approaches remain different. The supremum over $g \in L^2_m(\partial\mathcal{D})$ is indeed replaced by a maximum over a finite dimensional subspace of $L^2_m(\partial\mathcal{D})$ (see~\eqref{eq:infsupdiscretize}--\eqref{eq:supsquare}), which is itself approximated by~\eqref{eq:supsquare_practical}.

\begin{proof}[Proof of Proposition~\ref{prop:PsiuBoundaryPsiE}] 
Let us fix $\eps > 0$ and $\overline{A} \in \mathcal{S}_{\alpha, \beta}$.

We recall that $\mathcal{T}_\eps$, defined in Section~\ref{subsection:supproblem}, denotes the operator that maps $g \in L^2_m(\partial \mathcal{D})$ to $u_\eps(g) \in L^2_m(\partial \mathcal{D})$, the unique solution in $H^1_m(\mathcal{D})$ to~\eqref{eq:diffusive}. Similarly, we define the operator $\mathcal{T}_{\overline{A}}$ that maps $g\in L^2_m(\partial\mathcal{D})$ to $\overline{u}(g) \in L^2_m(\partial\mathcal{D})$, the unique solution in $H^1_m(\mathcal{D})$ to~\eqref{eq:diffusiveAbar}. Both operators (considered as operators from $L^2_m(\partial\mathcal{D})$ to $L^2_m(\partial\mathcal{D})$) are linear, self-adjoint and compact.

Consequently, the operator $\mathcal{H}_{\eps, \overline{A}} := \mathcal{T}_\eps - \mathcal{T}_{\overline{A}}$ (considered as an operator from $L^2_m(\partial\mathcal{D})$ to $L^2_m(\partial\mathcal{D})$) is linear, self-adjoint and compact. The eigenvalues of $\mathcal{H}_{\eps, \overline{A}}$ hence form a sequence of real numbers $\left\{ \mu_{\eps, \overline{A}}^j \right\}_j$ that converges to $0$. We can consider the eigenvalue $\mu_{\eps, \overline{A}}^{\text{sup}}$ with the largest absolute value, and an associated eigenmode $g_{\eps, \overline{A}}^{\text{sup}}$. We have
\begin{equation*}
  \left( \mathcal{T}_\eps - \mathcal{T}_{\overline{A}} \right) g_{\eps,\overline{A}}^{\text{sup}} = \mu_{\eps, \overline{A}}^{\text{sup}} \ g_{\eps, \overline{A}}^{\text{sup}}.
\end{equation*}
It is clear that $\left(\mu_{\eps, \overline{A}}^{\text{sup}}\right)^2$ is the largest eigenvalue of the operator $\mathcal{H}_{\eps, \overline{A}}^2 := \left(\mathcal{T}_\eps - \mathcal{T}_{\overline{A}} \right)^2$ (which is a diagonalizable operator when considered as an operator from $L^2_m(\partial\mathcal{D})$ to $L^2_m(\partial\mathcal{D})$). We hence have
\begin{align*}
  \Psi_\eps^{\text{ME}}(\overline{A})
  &= \sup_{\scriptsize \begin{aligned} &g \in L^2_m(\partial\mathcal{D}), \\[-0.5em] \|&g\|_{L^2(\partial\mathcal{D})}=1 \end{aligned}} \left| \mathcal{E}(A_\eps, g) - \mathcal{E}(\overline{A},g) \right| \\
  &= \sup_{g \in L^2_m(\partial\mathcal{D})} \frac{|\mathcal{E}(A_\eps, g) - \mathcal{E}(\overline{A},g)|}{\|g\|_{L^2(\partial\mathcal{D})}^2} \\
  &= \frac{1}{2} \, \sup_{g \in L^2_m(\partial\mathcal{D})} \left| \frac{\int_{\partial \mathcal{D}} g \left(\mathcal{T}_\eps - \mathcal{T}_{\overline{A}}\right) g}{\|g\|_{L^2(\partial\mathcal{D})}^2} \right| \\
  &= \frac{1}{2} \left| \mu_{\eps, \overline{A}}^{\text{sup}} \right|,
\end{align*}
and
\begin{align*}
  \Psi_\eps^{\text{MS}}(\overline{A})
  &= \sup_{\scriptsize \begin{aligned} &g \in L^2_m(\partial\mathcal{D}), \\[-0.5em] \|&g\|_{L^2(\partial\mathcal{D})}=1 \end{aligned}} \|u_\eps(g) - u(\overline{A},g)\|_{L^2(\partial \mathcal{D})} \\
  &= \sup_{g \in L^2_m(\partial \mathcal{D})} \frac{\|u_\eps(g) - u(\overline{A},g)\|_{L^2(\partial\mathcal{D})}}{\|g\|_{L^2(\partial\mathcal{D})}} \\
  &= \sup_{g \in L^2_m(\partial \mathcal{D})} \sqrt{\frac{\int_{\partial\mathcal{D}} g \left(\mathcal{T}_\eps - \mathcal{T}_{\overline{A}}\right)^2 g}{\|g\|_{L^2(\partial\mathcal{D})}^2}} \qquad \text{[using that $\mathcal{T}_\eps - \mathcal{T}_{\overline{A}}$ is self-adjoint]}\\
  &= \sqrt{\sup_{g \in L^2_m(\partial\mathcal{D})} \frac{\int_{\partial\mathcal{D}} g\left(\mathcal{T}_\eps - \mathcal{T}_{\overline{A}}\right)^2 g}{\|g\|_{L^2(\partial\mathcal{D})}^2}}\\
  &= \sqrt{\left(\mu_{\eps, \overline{A}}^{\text{sup}}\right)^2} = \left|\mu_{\eps, \overline{A}}^{\text{sup}}\right|.
\end{align*}
Proposition~\ref{prop:PsiuBoundaryPsiE} follows.
\end{proof}

\section{Numerical experiments} \label{section:section5}

We now describe our numerical approach to solve~\eqref{eq:infsup} in practice. In Section~\ref{subsection:algorithm}, we detail the implementation of our algorithm. In Section~\ref{subsection:testcases}, we introduce two test cases and we present the associated numerical results in Section~\ref{subsection:numericalresults}. For simplicity, we restrict ourselves to numerical experiments in dimension $d=2$. There is no restriction (besides more expensive computational costs) to put our approach in practice in higher dimensional settings (the theoretical analysis presented in Section~\ref{section:section3big} being itself valid in any ambient dimension).

\subsection{The algorithm} \label{subsection:algorithm}

To solve~\eqref{eq:infsup}, we start by approximating the supremum~\eqref{eq:supproblem} over all boundary conditions in $L^2_m(\partial \mathcal{D})$ by a maximum over the space spanned by a finite number of boundary conditions (see Section~\ref{section:changingsupbymax}). We then perform a gradient descent to approximate the infimum problem~\eqref{eq:infproblem} (see Section~\ref{section:onlinestage}). Each step $k$ of the gradient descent requires to solve~\eqref{eq:diffusiveAbar} for a new coefficient $\overline{A}$. We highlight that a \textit{coarse} mesh is sufficient to obtain an accurate approximation of $\overline{u}$, since $\overline{A}$ is constant.

\subsubsection{Approximating the supremum by a maximum} \label{section:changingsupbymax}

First, we recall that, in the general case, computing the supremum~\eqref{eq:supproblem} over all boundary conditions in $L^2_m(\partial \mathcal{D})$ is not possible. This supremum is therefore approximated by a maximum over an appropriate finite dimensional subspace $V_n^P(\partial\mathcal{D})$ of $L^2_m(\partial \mathcal{D})$.

The study conducted in Section~\ref{subsection:supproblem} suggests to use the subspace $V_n^P(\partial\mathcal{D})$ defined by~\eqref{eq:VnP} and generated by the $P \geq \frac{d(d+1)}{2}$ first eigenvectors $\{ \phi_p \}_{1\leq p \leq P}$ of the operator $\mathcal{R}$ defined by~\eqref{eq:laplacianinverse}.
In our numerical tests, these eigenmodes are computed using the ARPACK package in FreeFem++~\cite{hecht_freefem}. Our simulations are performed in dimension $d=2$, and the number $P$ of boundary conditions varies from 3 to 5, depending on the value of $\eps$ (see Section~\ref{subsection:numericalresults}). Problem~\eqref{eq:infsup} is now replaced by
\begin{equation} \label{eq:infsupdiscretize}
  \inf_{\overline{A}\in \mathbb{R}^{d\times d}_{\text{sym}}} \Psi_{\eps, P}(\overline{A}),
\end{equation}
where we recall that $\mathbb{R}^{d\times d}_{\text{sym}}$ is the space of symmetric $d \times d$ matrices, and where we have introduced
\begin{equation} \label{eq:supsquare}
  \Psi_{\eps, P}(\overline{A}) = \max_{g \in V_n^P(\partial \mathcal{D})} \psi_\eps(\overline{A}, g)^2 = \max_{g \in V_n^P(\partial \mathcal{D})} \left( \mathcal{E}(A_\eps,g) - \mathcal{E}(\overline{A}, g) \right)^2.
\end{equation}
As mentioned in Remark~\ref{remark:spaceA}, note that we relax the constraints on $\overline{A}$ in~\eqref{eq:infsupdiscretize}, and look for $\overline{A}$ in the space of constant symmetric matrices rather than imposing that $\overline{A}$ belongs to $\mathcal{S}_{\alpha, \beta}$. Note also that, in~\eqref{eq:supsquare}, we optimize upon the square (rather than the absolute value as in~\eqref{eq:infsup}) of the energy difference (we have thus changed the definition of $\Psi_{\eps,P}$ but kept that of $\psi_\eps$). This is simply a numerical trick making the optimization easier.

\subsubsection{Gradient descent} \label{section:onlinestage}

To solve~\eqref{eq:infsupdiscretize}, we perform a gradient descent. It is typically initiated on the mean value of $A_{\text{per}}$. Each iterate $\overline{A}^{k+1}$ is obtained from the previous iterate by
\begin{equation} \label{eq:gradient_descent}
  \overline{A}^{k+1} = \overline{A}^{k} - \mu \, \nabla_{\overline{A}}\Psi_{\eps, P}(\overline{A}^k).
\end{equation}
The gradient $\nabla_{\overline{A}}\Psi_{\eps, P}(\overline{A}^k)$ is computed using the adjoint method. Its expression is
\begin{equation*}
  \nabla_{\overline{A}} \Psi_{\eps, P}(\overline{A}^k) = - \left( \mathcal{E}(A_\eps,g^k) - \mathcal{E}(\overline{A}^k, g^k) \right) \left( (2 - \delta_{ij}) \int_{\mathcal{D}} \partial_i u(\overline{A}^k, g^k) \, \partial_j u(\overline{A}^k, g^k) \right)_{1 \leq i \leq j \leq d},
\end{equation*}
where $g^k$ denotes the argmax of $\Psi_{\eps, P}(\overline{A}^k)$ and $\delta_{ij}$ denotes the Kronecker symbol ($\delta_{ij} = 1$ if $i=j$ and $0$ otherwise). 

Each iterate requires to identify $g^k$, then to compute the field $u(\overline{A}^k, g^k)$ and the difference in energy $\mathcal{E}(A_\eps,g^k)-\mathcal{E}(\overline{A}^k, g^k)$. The task is expensive since computing these quantities essentially requires computing $u(\overline{A}^k,\phi_p)$ for all $1\leq p \leq P$. Nevertheless, since $\overline{A}^k$ is constant, it is sufficient to use a coarse mesh to perform these computations. In practice, we use a mesh of size $H=0.05$.

For a new iterate $\overline{A}^k$, we thus start by computing $u(\overline{A}^k, \phi_p)$ for all $1\leq p \leq P$ on a coarse mesh of size $H$. Then we find the argmax $\dis g^k = \sum_{p=1}^P c_p^k \, \phi_p \in V_n^P(\partial\mathcal{D})$ to $\Psi_{\eps,P}(\overline{A}^k)$. Using the linearity of~\eqref{eq:diffusive} and~\eqref{eq:diffusiveAbar}, identifying the coefficients $\{ c_p^k \}_{1 \leq p \leq P}$ amounts to solving the problem
\begin{equation*}
  \max_{g \in V_n^P(\partial\mathcal{D})} \psi_\eps(\overline{A}^k, g) = \max_{\sum_{p=1}^P c_p^2 = 1} \left| \sum_{1 \leq p,q \leq P} c_p \, \left[ M(\overline{A}^k) \right]_{pq} \, c_q \right|,
\end{equation*}
with the $P \times P$ matrix $M(\overline{A}^k)$ defined by 
\begin{equation} \label{eq:matrixMk}
  M(\overline{A}^k) = \left( \frac{1}{2} \int_{\partial \mathcal{D}} \phi_q \left( u_\eps(\phi_p) - u(\overline{A}^k, \phi_p) \right) \right)_{1\leq p,q \leq P}.
\end{equation}
This boils down to finding the eigenvector associated to the eigenvalue with largest absolute value of the symmetric matrix $M(\overline{A}^k)$ (in practice, we have always observed this eigenvalue to be simple). The difference $\dis \left| \mathcal{E}(A_\eps, g^k) - \mathcal{E}(\overline{A}^k, g^k) \right|$ is equal to the absolute value of the eigenvalue associated to $g_k$. We then write that $\dis u(\overline{A}^k, g^k) = \sum_{p=1}^P c_p^k \, u(\overline{A}^k, \phi_p)$ and we are in position to compute the value of the terms $\dis \left\{ \int_{\mathcal{D}} \partial_i u(\overline{A}^k, g^k) \, \partial_j u(\overline{A}^k, g^k) \right\}_{1 \leq i \leq j \leq d}$. Once this has been done, the direction of descent $-\nabla_{\overline{A}}\Psi_{\eps, P}(\overline{A}^k)$ can be computed.

In our experiments, the step $\mu$ is selected with a line search algorithm. The Armijo rule is used with the parameter $m=0.1$ (chosen empirically). In practice, convergence of the algorithm~\eqref{eq:gradient_descent} is typically observed after a few dozens of iterations (see Figure~\ref{fig:combined}).

\begin{figure}[htbp]
\centering
     
\begin{subfigure}[t]{0.48\textwidth}
  \centering
  \begin{tikzpicture}[scale=0.9]
    \begin{axis}[
        xlabel=$k$,
        legend pos=north east,
        grid=both,
        xmin=0, xmax=50,
        ymin=0., ymax=0.2,
      ]       
      \addplot[ 
        mark=+,
        blue, 
        mark size=3pt,
        line width=1pt
      ] file {Data/costfunction_InfSumEnergyNeumannGradientDescent_OscillatingCase_highcontrast_eps0.05_P3_noHessian_Armijo_m1=0.1_H=0.005_projected_caseperiodic.txt};
      \legend{$\Psi_{\eps, P}(\overline{A}^k)$};
    \end{axis}
  \end{tikzpicture}
  \caption{Evolution of $\Psi_{\eps, P}(\overline{A}^k)$ with iteration $k$.}
  \label{fig:costfunck}
\end{subfigure}
\hfill
\begin{subfigure}[t]{0.48\textwidth}
  \centering
  \begin{tikzpicture}[scale=0.9]
    \begin{axis}[
        xlabel=$k$,
        legend style={at={(axis cs:48,1)},anchor=south east},
        grid=both,
        xmin=0, xmax=50,
        ymin=-1., ymax=24,
      ]       
      \addplot[ 
        mark=+,
        blue, 
        mark size=3pt,
        line width=1pt
      ] table [x index=0, y index=1] {Data/coefficient_InfSumFieldEnergyNeumannGradientDescentPartial_StationnaryErgodicrmr_eps0.05_P3_noHessian_Armijo_m1=0.1_H=0.005.txt};
      \addplot[ 
        mark=+,
        purple, 
        mark size=3pt,
        line width=1pt
      ] table [x index=0, y index=3] {Data/coefficient_InfSumFieldEnergyNeumannGradientDescentPartial_StationnaryErgodicrmr_eps0.05_P3_noHessian_Armijo_m1=0.1_H=0.005.txt};
      \addplot[ 
        mark=+,
        orange, 
        mark size=3pt,
        line width=1pt
      ] table [x index=0, y index=2] {Data/coefficient_InfSumFieldEnergyNeumannGradientDescentPartial_StationnaryErgodicrmr_eps0.05_P3_noHessian_Armijo_m1=0.1_H=0.005.txt};
      
      \legend{$\overline{A}_{11}$, $\overline{A}_{22}$, $\overline{A}_{12}$};
    \end{axis}
  \end{tikzpicture}
  \caption{Evolution of the components of $\overline{A}^k$ with iteration $k$.}
  \label{fig:Abark}
\end{subfigure}

\caption{Convergence history of the algorithm for $\eps=0.05$ and $P=3$.}
\label{fig:combined}
\end{figure}
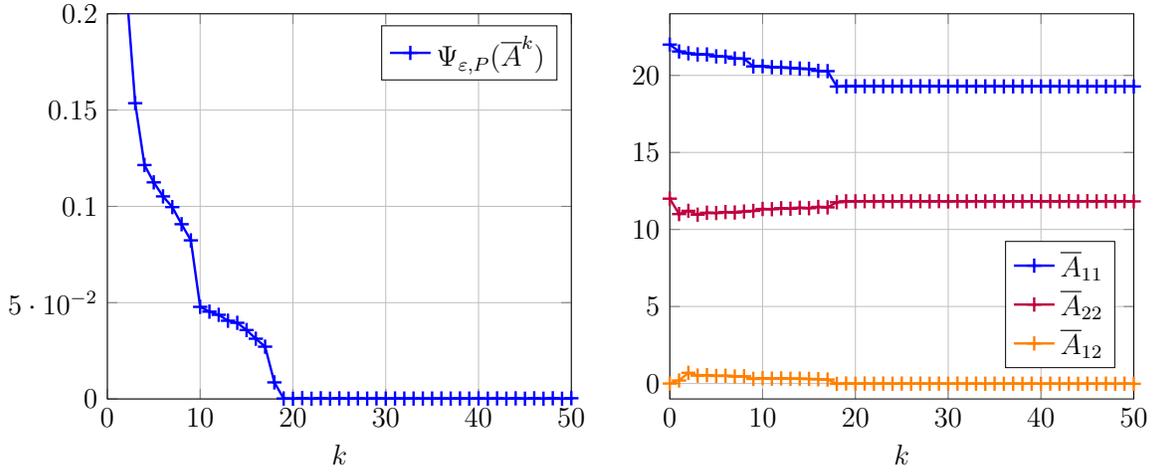

\begin{remark}
Note that, to compute the matrix $M(\overline{A}^k)$ defined in~\eqref{eq:matrixMk}, we need the value of $\dis \int_{\partial \mathcal{D}} \phi_q \, u_\eps(\phi_p)$ for each couple $(p,q) \in \llbracket 1, P\rrbracket^2$. In the practical implementation of our inversion procedure in an experimental context, these energies, for the case $q=p$, would be given by actual measurements. The cross terms with $p \neq q$ can then be computed by a polarization formula:
\begin{equation*}
  \int_{\partial \mathcal{D}} \phi_q \, u_\eps(\phi_p) = - \Big( \mathcal{E}(A_\eps,\phi_p + \phi_q) - \mathcal{E}(A_\eps,\phi_p) - \mathcal{E}(A_\eps,\phi_q) \Big).
\end{equation*}
In practice, computing the cross terms can be tedious. In our numerical experiments, replacing $\Psi_{\eps, P}$ by $\Psi_{\eps, P}^\Sigma$, defined as
\begin{equation} \label{eq:supsquare_practical}
  \Psi_{\eps, P}^\Sigma(\overline{A}) = \sum_{p=1}^P \psi_\eps(\overline{A}, \phi_p)^2,
\end{equation}
yields equivalent results and does not require to precompute the cross terms $\dis \int_{\partial \mathcal{D}} \phi_p \, u_\eps(\phi_q)$ (for any $1 \leq p < q \leq P$), nor to optimize with respect to some $g \in L^2_m(\partial \mathcal{D})$. This is what we have done.
\end{remark}

\subsection{Test cases} \label{subsection:testcases}

In this section, we introduce our two test cases, on the domain $\mathcal{D} = (0,1)^2$. 

\subsubsection{Periodic setting} \label{subsection;testcaseperiodic}

We consider the coefficient
\begin{equation*}
  A_\eps(x,y) = A_{\text{per}}\left( \frac{x}{\eps}, \frac{y}{\eps}\right),
\end{equation*}
where $A_{\text{per}}$ is the $\mathbb{Z}^2$-periodic (symmetric matrix valued) coefficient given (see Figure~\ref{fig:per_coeff_11} for some illustration) by 
\begin{equation} \label{eq:A11A22A12per}
\begin{aligned}
  [A_{\text{per}}(x,y)]_{11} &= 22 + 10 \left( \sin(2\pi x) + \sin(2\pi y) \right),\\
  [A_{\text{per}}(x,y)]_{22} &= 12 + 2 \left( \sin(2\pi x) + \sin(2\pi y) \right),\\
  [A_{\text{per}}(x,y)]_{12} &= [A_{\text{per}}(x,y)]_{21} = 0.
\end{aligned}
\end{equation}
The associated homogenized coefficient can be computed using the classical formula~\eqref{eq:astar} involving the correctors defined by~\eqref{eq:corrector}. We obtain
\begin{equation} \label{eq:A_star_cas_per}
  [A_\star]_{11} \approx 19.3378, \qquad [A_\star]_{22} \approx 11.8312, \qquad [A_\star]_{12} = [A_\star]_{21} \approx 0,
\end{equation}
where the corrector functions have been computed on a mesh with discretization parameter $h_\star \approx 5 \times 10^{-4}$. 
We will consider these values of $[A_\star]_{ij}$ as our reference values.

\begin{figure}[htbp]
\centering
\includegraphics[width=5.5cm]{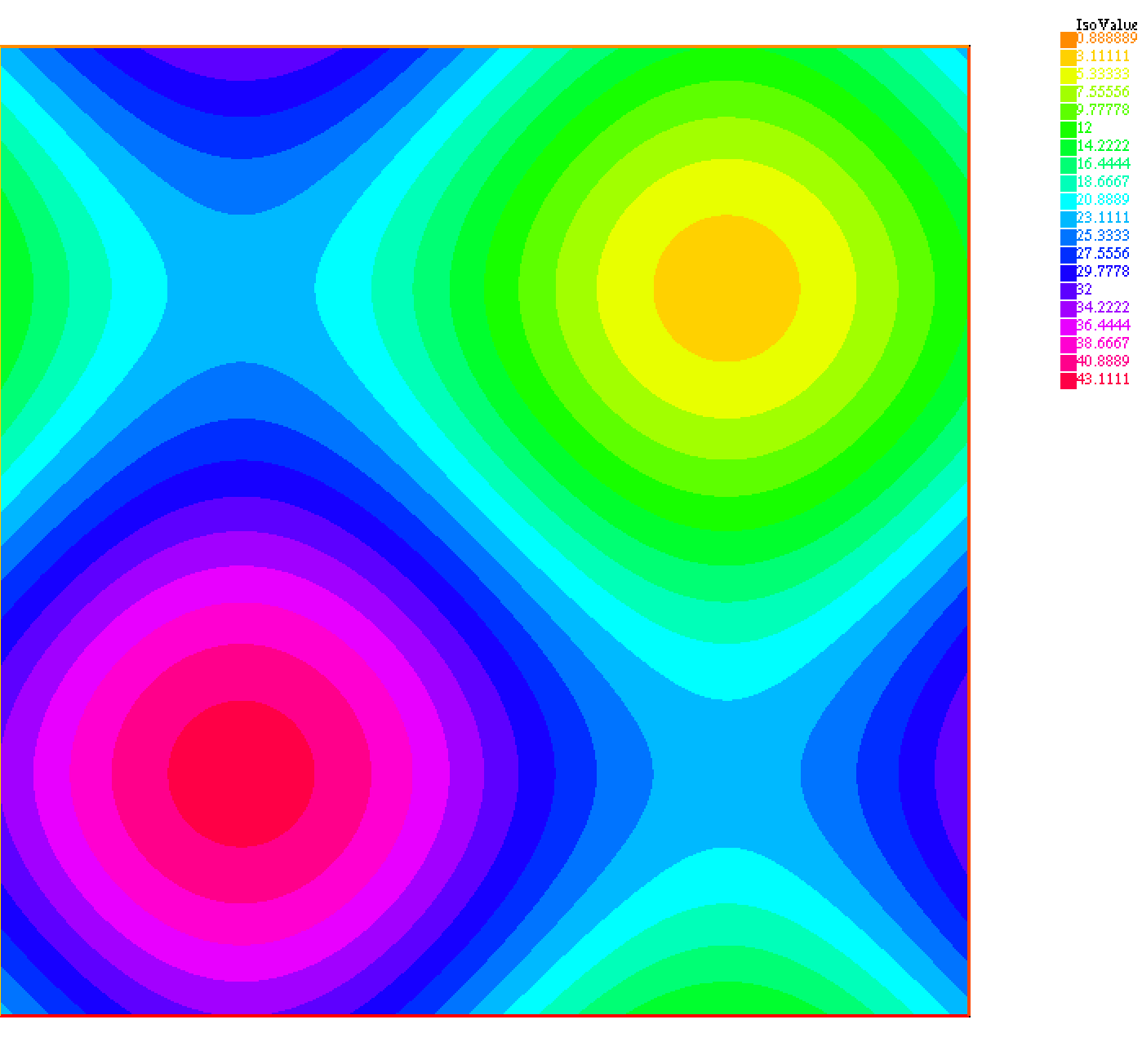}
\caption{Component $11$ of $A_{\text{per}}$ defined in~\eqref{eq:A11A22A12per} and displayed on the periodic cell $(0,1)^2$.}
\label{fig:per_coeff_11}
\end{figure}

\subsubsection{Random stationary setting} \label{subsection:testcaserandom} 

We alternatively consider here a random (stationary) setting. Like in the periodic setting, the problem has a deterministic homogenized limit with constant coefficients. The setting therefore falls within the theoretical setting considered in Section~\ref{section:section3} (see Remark~\ref{rem:extension}). In addition, in general, identifying this homogenized matrix requires expensive numerical computations. It is therefore interesting to design variants of our deterministic formulation~\eqref{eq:infsup} and investigate their ability to provide an accurate effective operator. We consider here a particular example, for which the homogenized matrix is analytically known. 

Let $(\Omega, \mathcal{F}, \mathbb{P})$ be a probability space and let us consider the two-dimensional random checkerboard coefficient defined, for all $x \in \mathcal{D}$ and all $\omega \in \Omega$, by
\begin{equation*}
A_\eps(x, \omega) = a_{\text{rand}}\left(\frac{x}{\eps}, \omega\right) \text{Id},
\end{equation*}
where $a_{\text{rand}}$ is the discrete random stationary (scalar valued) coefficient defined by
\begin{equation*}
a_{\text{rand}}(x,\omega) = \sum_{k \in \mathbb{Z}^2} \mathbb{1}_{Q+k}(x) \, X_k(\omega),
\end{equation*}
where $Q=(0,1)^2$ and $X_k$ are i.i.d. random variables following a Bernoulli law: $\mathbb{P}(X_k=4) = \mathbb{P}(X_k=16) = 1/2$ (see Figure~\ref{fig:sto_coeff} for some illustration).

\begin{figure}[htbp]
  \centering
  \includegraphics[width=7.5cm]{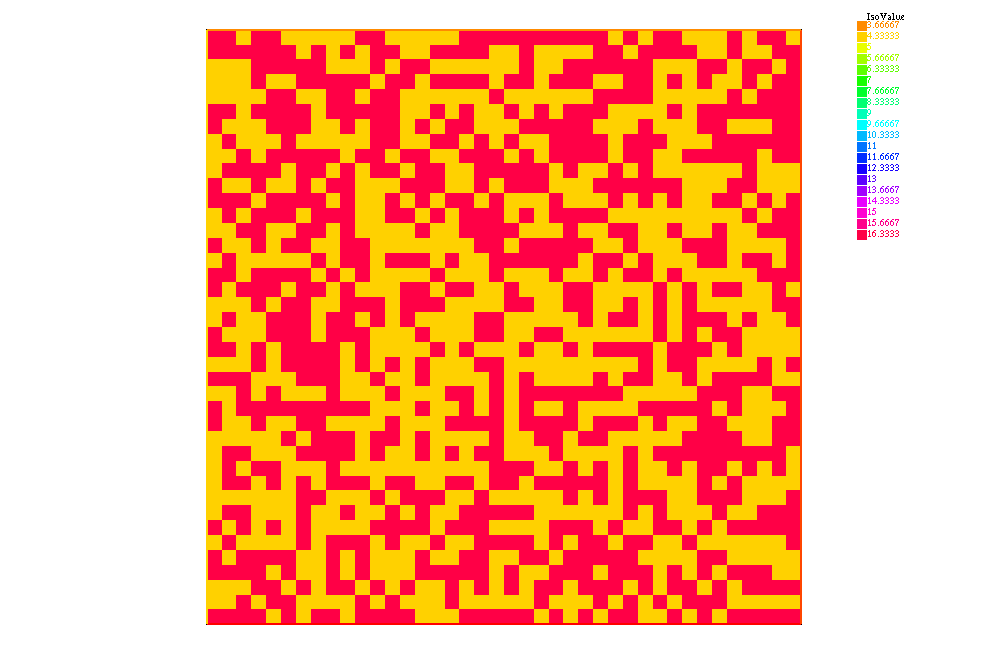}
  \includegraphics[width=7.5cm]{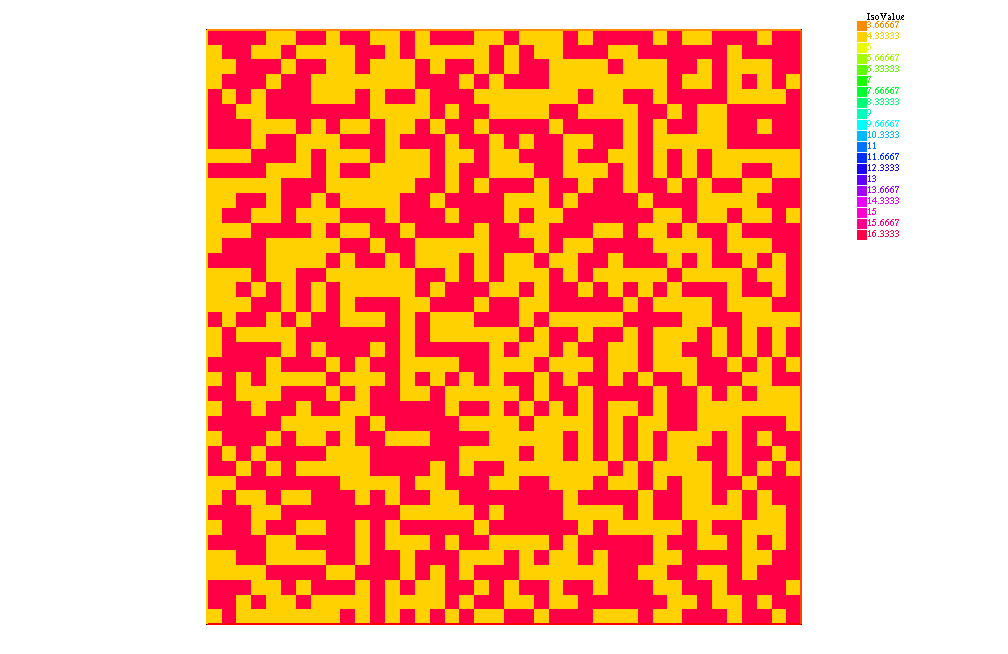}
  \caption{Two realizations of $a_{\text{rand}}(\cdot, \omega)$ restricted to the square $(-20, 20)^2$.}
  \label{fig:sto_coeff}
\end{figure}

Homogenization theory applies and it can be shown (see e.g.~\cite{papanicolaou1979boundary}) that, for any $g \in L^2_m(\partial \mathcal{D})$, the random solution $u_\eps$ (we again enforce that the mean of $u_\eps(\cdot,\omega)$ on $\partial \mathcal{D}$ vanishes) to 
\begin{equation} \label{eq:diffusive_sto}
  \left\{
  \begin{aligned}
    -\text{div} \left(A_\eps(x, \omega) \nabla u_\eps(x,\omega)\right) = 0 \quad \text{a.s. in $\mathcal{D}$},\\
    \left(A_\eps(x, \omega) \nabla u_\eps(x,\omega)\right) \cdot n = g \quad \text{a.s. on $\partial\mathcal{D}$},
  \end{aligned}
  \right.
\end{equation}
converges, weakly in $L^2(\Omega, H^1(\mathcal{D}))$, to the solution to the deterministic homogenized problem~\eqref{eq:eqstar}, where the homogenized coefficient is analytically known:
\begin{equation*}
  A_\star = \sqrt{4\times16} \ \text{Id} = 8 \, \text{Id}.
\end{equation*}

In this random setting, the strategy~\eqref{eq:problemofkunli} rewrites
\begin{equation*} 
  \inf_{\overline{A} \in \mathcal{S}_{\alpha, \beta}} \ \sup_{\scriptsize \begin{aligned} &g \in L^2_m(\partial \mathcal{D}), \\[-0.5em] \|&g\|_{L^2(\partial \mathcal{D})}=1 \end{aligned}} \left\| \mathbb{E}\left[u_\eps(g, \omega)\right] - u(\overline{A}, g) \right\|_{L^2(\mathcal{D})},
\end{equation*}
the strategy~\eqref{eq:infsup} rewrites
\begin{equation*} 
  \inf_{\overline{A} \in \mathcal{S}_{\alpha, \beta}} \ \sup_{\scriptsize \begin{aligned} &g \in L^2_m(\partial \mathcal{D}), \\[-0.5em] \|&g\|_{L^2(\partial \mathcal{D})}=1 \end{aligned}} \left| \mathbb{E}\big[\mathcal{E}(A_\eps(\cdot, \omega), g)\big] - \mathcal{E}(\overline{A}, g) \right|,
\end{equation*}
while~\eqref{eq:supproblemboundary} rewrites
\begin{equation*} 
  \inf_{\overline{A} \in \mathcal{S}_{\alpha, \beta}} \ \sup_{\scriptsize \begin{aligned} &g \in L^2_m(\partial \mathcal{D}), \\[-0.5em] \|&g\|_{L^2(\partial \mathcal{D})}=1 \end{aligned}} \left\| \mathbb{E}\left[u_\eps(g, \omega)\right] - u(\overline{A}, g) \right\|_{L^2(\partial\mathcal{D})}.
\end{equation*}
Note that these formulations are arbitrary. Other choices are possible, e.g. placing the expectation outside the supremum. The choice of the above formulations is motivated by the fact that it significantly reduces the computational costs. Indeed, the minimization and/or the maximization are computed only once (rather than for each $\omega$).

\subsection{Numerical results} \label{subsection:numericalresults}


We apply and compare the strategies mentioned above for the two test cases introduced in Section~\ref{subsection:testcases}. We compare the results
\begin{itemize}
\item of the strategy~\eqref{eq:infsup}, referred to as strategy \textit{ME} (``Measurement in Energy'');
\item of the strategy presented in~\cite{le2013approximation}, redefined in~\eqref{eq:problemofkunli} and referred to as strategy \textit{MV} (``Measurement in Volume''); 
\item of the strategy consisting in minimizing~\eqref{eq:supproblemboundary}, referred to as strategy \textit{MS} (``Measurement on Surface'').
\end{itemize}
The above three strategies provide an effective coefficient that a priori depends on $\eps$. It is also interesting to consider the homogenized coefficient and to compare, in the regime when $\eps$ is not asymptotically small, the effective models obtained using the above strategies with the homogenized model. We expect (and this is indeed the case, as shown below) that our effective models perform better than the homogenized model, since they are built using solutions of the actual heterogeneous model and thus encode in some way the precise value of the scale $\eps$. 

The comparison between the various strategies is performed first in a periodic case (in Section~\ref{sec:periodic_case}), and next in a random stationary case (in Section~\ref{sec:random_case}). 
Our aim with the numerical tests reported below is
\begin{itemize}
\item to investigate the impact of considering a finite-dimensional space (of dimension $P$) for the boundary conditions, rather than the space $L^2_m(\partial \mathcal{D})$ as in the formulation~\eqref{eq:infsup}. We will see below that the choice of $P$ suggested in Section~\ref{sec:precomput} leads to satisfying results, in the sense that the practically computed effective coefficient does converge to the homogenized coefficient when $\eps \to 0$, and that, for larger values of $\eps$, the effective operator indeed turns out to be an accurate approximation of the oscillatory operator. In addition, we briefly investigate (in the periodic case) a choice for the boundary conditions alternative to the one suggested above (namely considering the eigenmodes of the operator $\mathcal{R}$ defined in~\eqref{eq:laplacianinverse}), which turns out to lead to less accurate results. This confirms our initial way of choosing the boundary conditions.
\item investigate the accuracy with which our effective operator approximates the oscillatory operator, in particular in the regime when $\eps$ is not small. We then consider a large number $Q \gg P$ of boundary conditions, and monitor the difference between $u_\eps(g)$ and $u(\overline{A}, g)$ for any $g$ in the generated space of dimension $Q$. A typical result is that, for $\eps = 0.25$ (to be compared with the size of $\mathcal{D}$), the relative error (in $L^2(\mathcal{D})$ norm) between $u_\eps(g)$ and $u(\overline{A}, g)$ is of the order of 20\%.
\end{itemize}
We eventually investigate the robustness of our approach with respect to noise (in Section~\ref{sec:noise_general}), seen as either a disturbance or a way to improve the quality of the approach.

\subsubsection{Precomputations} \label{sec:precomput}

We assume that we are given the observables, namely the fields $u_\eps(g)$ in the whole domain $\mathcal{D}$, their restriction to $\partial \mathcal{D}$ or the energies $\mathcal{E}(A_\eps,g)$. We recall that, in the practical implementation of our algorithm, the latter quantities could be given by actual measurements. In the simulations presented below, these quantities are obtained as follows.

We compute the solutions $u_\eps(\phi_p)$ to~\eqref{eq:diffusive} for any $1\leq p \leq P$ using a sufficiently fine mesh that captures all the variations of the coefficient $A_\eps$. To compute the energy, we simply compute the integral~\eqref{eq:energy} on the fine mesh. In practice, we use a mesh size $h\ll \eps$, namely $h = \eps/r$ with $r=40$ in the deterministic case and $r=20$ in the random case. This choice is motivated by the fact that, in the random case, we consider many realizations of the coefficient $a_{\text{rand}}(\cdot, \omega)$ (and we thus compute, for each boundary condition, many solutions to~\eqref{eq:diffusive_sto}), in order to approximate the expectation of the solutions (or of the energy). We thus increase the mesh size in order to reduce the computational time.

In the random case, expectations are computed using $M_1=40$ independent realizations of $a_{\text{rand}}(\cdot, \omega)$. Furthermore, confidence intervals at 95\% are computed using $M_2=40$ independent batches of $M_1=40$ independent realizations of $a_{\text{rand}}(\cdot, \omega)$. 

As far as the number $P$ of boundary conditions is concerned, we choose $P=3$ for $\eps < 0.2$ and $P=5$ for $\eps \geq 0.2$ to compute the coefficient $\overline{A}^{\text{ME}}_\eps$. These values of $P$ are determined empirically, taking into account a balance between limiting the number of considered boundary conditions and maintaining the quality of the resulting effective coefficient. To ensure a fair comparison, we take the same value of $P$ to compute $\overline{A}^{\text{MS}}_\eps$ and $\overline{A}^{\text{MV}}_\eps$. 

\subsubsection{Periodic case} \label{sec:periodic_case}


We focus on the periodic setting presented in Section~\ref{subsection;testcaseperiodic}.

Figure~\ref{fig:absoluteerrorAbar} shows the relative error 
\begin{equation} \label{eq:errorAB}
  \text{Err}_\star(\overline{A}) = \sqrt{\frac{\sum_{1\leq i \leq j \leq d} \, (\overline{A}_{ij} - A_{\star,ij})^2}{\sum_{1\leq i \leq j \leq d} A_{\star,ij}^2}}
\end{equation}
between any effective coefficient $\overline{A} \in \left\{ \overline{A}^{\text{ME}}_\eps, \overline{A}^{\text{MS}}_\eps, \overline{A}^{\text{MV}}_\eps \right\}$ and the homogenized coefficient $A_\star$.

When $\eps \to 0$, we observe that each strategy recovers the homogenized coefficient $A_\star$. This consistency would be expected for $\overline{A}^{\text{ME}}_\eps$ (see Proposition~\ref{prop:quasiminimizers}), as well as for $\overline{A}^{\text{MS}}_\eps$ (see Section~\ref{section:section4}) and for $\overline{A}^{\text{MV}}_\eps$ (see~\cite{le2013approximation}), if they were defined through a maximisation over the infinite dimensional space $L^2_m(\partial \mathcal{D})$ of all possible boundary conditions, and is again observed here in practice. In addition, the convergence of $\overline{A}^{\text{ME}}_\eps$, $\overline{A}^{\text{MS}}_\eps$ and $\overline{A}^{\text{MV}}_\eps$ to $A_\star$ is observed to be of order $\eps^2$. 

We also observe that $\overline{A}^{\text{ME}}_\eps$ and $\overline{A}^{\text{MS}}_\eps$ slightly differ, although they are expected to coincide from a theoretical point of view when maximizing over all $g \in L^2_m(\partial \mathcal{D})$ (see Section~\ref{section:section4}). The discrepancy that we observe in practice can be attributed to the approximation of the supremum~\eqref{eq:supproblem} by a maximum over a finite dimensional space spanned by some boundary conditions. Indeed, the boundary conditions selected in Section~\ref{section:changingsupbymax} to build this finite dimensional space were identified under the assumptions that the coefficients $A_\star$ and $\overline{A}$ are scalar. This is not the case here (recall~\eqref{eq:A_star_cas_per}, that shows that there is a ratio of the order of two between the two diagonal coefficients of $A_\star$), and we presume that this particularity explains the mismatch of the two coefficients $\overline{A}^{\text{ME}}_\eps$ and $\overline{A}^{\text{MS}}_\eps$. Nonetheless, we emphasize that the search space $V_n^P(\partial\mathcal{D})$ remains sufficiently rich (even though it does not include the ideal boundary conditions) for our approach to yield an effective coefficient $\overline{A}^{\text{ME}}_\eps$ which is (in the regime $\eps \to 0$) an accurate approximation of $A_\star$. 


\begin{figure}[htbp]
    \centering
    \begin{tikzpicture}[scale=1]
        \begin{axis}[
            xlabel=$\eps$,
            legend pos=north west,
            grid=both,
            ymode=log,
            xmode=log,
            xmin=0.01, xmax=0.5,
            legend style={nodes={scale=0.75}},
        ]        
            \addplot[color=red, mark=square, line width=1pt] coordinates{
                (0.4,   0.34638745095639106) 
                (0.3,   0.23029414947793211) 
                (0.225, 0.20878247497704897) 
                (0.2,  0.027115601530760022) 
                (0.1,  0.008428116313046033) 
                (0.05, 0.002998342141527602) 
                (0.025, 0.0005017513817364534)
            };
    
            \addplot[color=orange, mark=square, line width=1pt] coordinates{
                (0.4,   0.16953188700806127) 
                (0.3,   0.10285898575445625) 
                (0.225, 0.09805878516980554) 
                (0.2,   0.015055602913202107)  
                (0.1,   0.0049949033878236935) 
                (0.05,  0.0015929649147040903) 
                (0.025, 0.000316979373234581)
            };
             
            \addplot[color=brown, mark=square, line width=1pt] coordinates{
                (0.4,   0.09990307247081354)
                (0.3,   0.05574351968168939)
                (0.225, 0.05691911125822236)
                (0.2,   0.0073557833962990165)
                (0.1,   0.001359781197003937)
                (0.05,  0.00031865735231028877)
            };
    
            \addplot[color=black, dashed, line width=1pt] coordinates{
                (0.2,0.08 )
                (0.1,0.02)
                (0.05,0.005)
                (0.025,0.00125)
            };
            \legend{$\overline{A}_\eps^{\text{ME}}$, $\overline{A}_\eps^{\text{MS}}$, $\overline{A}_\eps^{\text{MV}}$, $y \propto \eps^2$};
        \end{axis}
    \end{tikzpicture}
    \caption{Error~\eqref{eq:errorAB} between the homogenized coefficient $A_\star$ and the effective coefficients $\overline{A}_\eps^{\text{ME}}$, $\overline{A}_\eps^{\text{MS}}$ and $\overline{A}_\eps^{\text{MV}}$, as a function of $\eps$.}
    \label{fig:absoluteerrorAbar}
\end{figure}
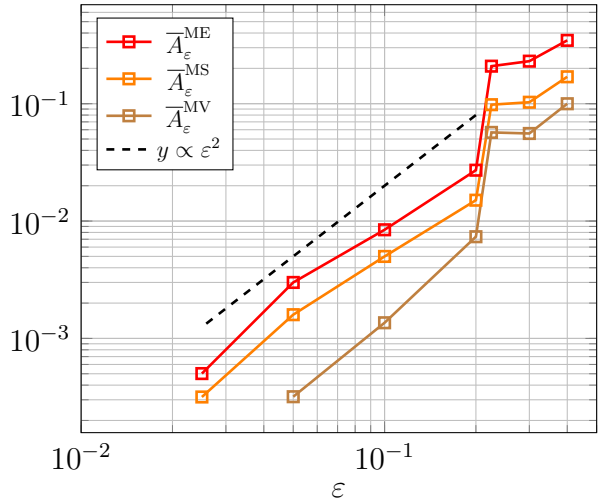

\medskip


In order to assess the quality of the effective operator (using the coefficient $\overline{A}$) in comparison to the oscillatory operator, we introduce the criterion
\begin{equation} \label{eq:criteriarelativ}
  \text{Err}_{\eps, Q}(\overline{A}) = \sup_{g \in V_n^Q(\partial\mathcal{D})} \frac{\left\| u_\eps(g) - u(\overline{A}, g) \right\|_{L^2(\mathcal{D})}}{\left\| u_\eps(g) \right\|_{L^2(\mathcal{D})}}.
\end{equation}
It represents the worst $L^2$ error on the whole domain $\mathcal{D}$. In practice, we choose $Q = 11 \gg P$, so that we test the effective coefficients $\overline{A}_\eps$ (computed using the data stemming from $P$ boundary conditions) on a large number of new boundary conditions. Figure~\ref{fig:absoluteerror} shows $\text{Err}_{\eps, Q}(\overline{A})$ for different values of $\eps$ and for the effective coefficients $\overline{A}^{\text{ME}}_\eps$, $\overline{A}^{\text{MS}}_\eps$ and $\overline{A}^{\text{MV}}_\eps$ (we also include the homogenized coefficient $A_\star$ in our comparison, to investigate whether, when $\eps$ is not asymptotically small, our approach indeed provides a more accurate approximation of $u_\eps$ than the homogenization approach).

In the regime $\eps \to 0$, all approaches perform similarly and with the accuracy predicted by homogenization theory. For instance, $\text{Err}_{\eps, Q}(\overline{A})$ is of the order of 0.03 (i.e. 3\%) for $\eps=0.025$, regardless of the choice $\overline{A} \in \left\{A_\star, \overline{A}_\eps^{\text{ME}}, \overline{A}_\eps^{\text{MS}}, \overline{A}_\eps^{\text{MV}} \right\}$. It is noteworthy that, among all the coefficients considered, only the computation of $A_\star$ relies on the full knowledge of the microstructure $A_\eps$ (whereas the other effective coefficients only depend on the knowledge of the energy or of the solution field). We also emphasize that computing $\overline{A}_\eps^{\text{ME}}$ requires much coarser data than computing $\overline{A}_\eps^{\text{MS}}$ or $\overline{A}_\eps^{\text{MV}}$.
 
For larger values of $\eps$ (say $\eps > 0.1$), Figure~\ref{fig:absoluteerror} shows that the effective coefficients $\overline{A}^{\text{ME}}_\eps, \overline{A}^{\text{MS}}_\eps, \overline{A}^{\text{MV}}_\eps$ outperform those non calibrated for this regime of oscillation. Typically, the homogenized coefficient $A_\star$ yields an approximation error that is $5\%$ to $10\%$ larger than the error associated with the coefficients obtained through the other approaches (for instance, for $0.2 \leq \eps \leq 0.3$, the error~\eqref{eq:criteriarelativ} is of the order of 30\% with the homogenized coefficient $A_\star$ while it is of the order of 20\% for the effective coefficients). While $\overline{A}^{\text{MV}}_\eps$ is surprisingly slightly less accurate, coefficients $\overline{A}^{\text{ME}}_\eps$ and $\overline{A}^{\text{MS}}_\eps$ yield the best approximations.



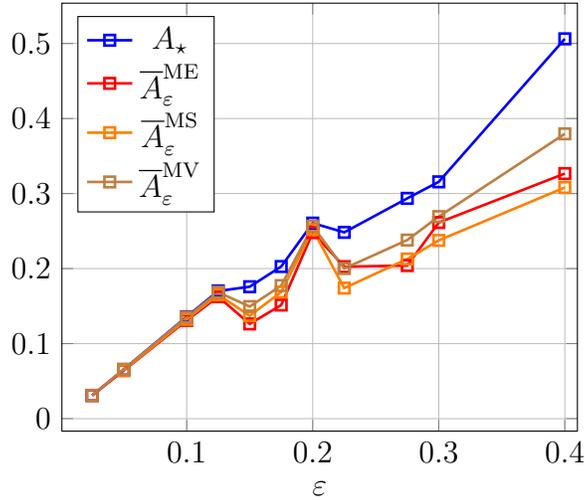
\begin{figure}[htbp]
    \centering
    \begin{tikzpicture}[scale=1]
      \begin{axis}[
          xlabel=$\eps$,
          legend pos=north west,
          grid=both,
          xmin=0.001, xmax=0.41,
      ]        
      \addplot[color=blue, mark=square, line width=1pt] coordinates{
            (0.4,0.506197428387412)
            (0.3,0.31561411914825654)
            (0.275,0.2935850458478148)
            (0.225,0.24828552577733126)
            (0.2,0.2605886663215005)
            (0.175,0.20274077516587577)
            (0.15,0.17581661190570572)
            (0.125,0.1700830909963723)
            (0.1,0.1358242935514575)
            (0.05,0.06554926904942866)
            (0.025, 0.03089304389077087)
      };
          \addplot[color=red, mark=square, line width=1pt] coordinates{
            (0.4,0.3266232129720646)
            (0.3,0.261359997688999)
            (0.275,0.20405252168941213)
            (0.225,0.20251251182164542)
            (0.2,0.24795174182750587)
            (0.175,0.1513967014973608)
            (0.15,0.12627058861298684)
            (0.125,0.16247646733830642)
            (0.1,0.13095018539702805)
            (0.05,0.06402922015084547)
            (0.025, 0.031050925879217033) 
          };
          \addplot[color=orange, mark=square, line width=1pt] coordinates{
            (0.4,0.30806588554013464)
            (0.3,0.23758092406765455)
            (0.275,0.2127589222835853)
            (0.225,0.17399570083452778)
            (0.2,0.25210277602508124)
            (0.175,0.16877323955397405)
            (0.15,0.13712613564009557)
            (0.125,0.16534044821794183)
            (0.1,0.13248642594234336)
            (0.05,0.06448248921807499)
          };
          \addplot[color=brown, mark=square, line width=1pt] coordinates{
            (0.4,0.3796090753977046)
            (0.3,0.2693295207273974)
            (0.275,0.23801183939840087)
            (0.225,0.20009860750722516)
            (0.2,0.2564839280834162)
            (0.175,0.17735686148082627)
            (0.15,0.14938039769515002)
            (0.125,0.16849617897416577)
            (0.1,0.13514873074139572)
            (0.05,0.0657468555786141)
            (0.025, 0.03089304389077087)
  
          };
          \legend{$A_\star$, $\overline{A}_\eps^{\text{ME}}$, $\overline{A}_\eps^{\text{MS}}$, $\overline{A}_\eps^{\text{MV}}$};
      \end{axis}
    \end{tikzpicture}
    \caption{Error~\eqref{eq:criteriarelativ} (computed with $Q=11$) for $\overline{A} \in \left\{ A_\star, \overline{A}_\eps^{\text{ME}}, \overline{A}_\eps^{\text{MS}}, \overline{A}_\eps^{\text{MV}} \right\}$, as a function of $\eps$.}
    \label{fig:absoluteerror}
\end{figure}

\medskip

We conclude our investigation of the periodic test case by temporarily considering the use of alternative boundary conditions for the construction of the effective coefficient. Instead of considering the first eigenmodes $\dis \{ \phi_i \}_{1 \leq i \leq d(d+1)/2}$ of the operator $\mathcal{R}$ defined in~\eqref{eq:laplacianinverse} (as we explained in Section~\ref{subsection:supproblem}), we consider the boundary conditions $\dis \{ \widetilde{\phi}_{ij} \}_{1 \leq i \leq j \leq d}$ defined by
$$
\widetilde{\phi}_{ij} = e_{ij} \cdot n \quad \text{with the notation} \quad e_{ij} = \frac{e_i + e_j}{2}
$$
for any $1 \leq i \leq j \leq d$, where $(e_1,\dots,e_d)$ denotes the canonical basis of $\mathbb{R}^d$ and $n$ is the outward unit normal vector to $\partial \mathcal{D}$. This choice is motivated by the fact that it may lead to easier-to-establish theoretical results (concerning in particular rates of convergence of the effective coefficient $\overline{A}_\eps$ to $A_\star$) for our approach (indeed, for any constant coefficient $\overline{A}$, the solution to~\eqref{eq:diffusiveAbar} with $g = \widetilde{\phi}_{ij}$ is analytically known and reads $\dis \overline{u}_{ij}(x) = \left( \overline{A}^{-1} e_{ij} \right) \cdot x$). 

After selecting these conditions, we consider $\widetilde{A}^{\text{ME}}_\eps$ obtained by solving~\eqref{eq:infsupdiscretize}--\eqref{eq:supsquare} where the space $V_n^P(\partial \mathcal{D})$ (normalized linear combinations of the $\phi_i$) is replaced by the space $\widetilde{V}_n^P(\partial \mathcal{D})$ of normalized linear combinations of the $\widetilde{\phi}_{ij}$. We then compare the quality of $\overline{A}^{\text{ME}}_\eps$ and $\widetilde{A}^{\text{ME}}_\eps$ by using the criteria~\eqref{eq:errorAB} and~\eqref{eq:criteriarelativ}. Results are presented on Figures~\ref{fig:absoluteerrorAbar_cohen} and~\ref{fig:absoluteerror_cohen}. As can be seen on Figure~\ref{fig:absoluteerrorAbar_cohen}, except for the smallest value of $\eps$ that we considered, $\widetilde{A}^{\text{ME}}_\eps$ is closer to $A_\star$ than $\overline{A}^{\text{ME}}_\eps$. The comparison of the solution fields is shown on Figure~\ref{fig:absoluteerror_cohen}. For small values of $\eps$ (say $\eps \leq 0.1$), we observe that both effective coefficients provide approximations that share the same accuracy (this was expected, since $\overline{A}^{\text{ME}}_\eps$ and $\widetilde{A}^{\text{ME}}_\eps$ are close to each other, as can be seen on Figure~\ref{fig:absoluteerrorAbar_cohen}). For larger values of $\eps$, the coefficient $\overline{A}^{\text{ME}}_\eps$ provides slightly more accurate approximations.

\begin{figure}
  \centering
  \begin{tikzpicture}[scale=1]
    \begin{axis}[
        xlabel=$\eps$,
        legend pos=north west,
        grid=both,
        ymode=log,
        xmode=log,
        xmin=0.01, xmax=0.5,
        legend style={nodes={scale=0.75}},
      ]        
      \addplot[color=red, mark=square, line width=1pt] coordinates{
        (0.4, 0.10774503492)
        (0.3, 0.056537630962)
        (0.225, 0.0577673301188)
        (0.2, 0.0134415307226)
        (0.1, 0.00584834914651)
        (0.05, 0.00233131887483)
        (0.025, 0.0007641960427696694)
      };
      \addplot[color=blue, mark=square, line width=1pt] coordinates{
        (0.4,   0.34638745095639106) 
        (0.3,   0.23029414947793211) 
        (0.225, 0.20878247497704897) 
        (0.2,  0.027115601530760022) 
        (0.1,  0.008428116313046033) 
        (0.05, 0.002998342141527602) 
        (0.025, 0.0005017513817364534) 
      };
    \addplot[color=black, line width=1pt] coordinates{
      (0.2,0.02)
      (0.1,0.01)
      (0.05,0.005)
      (0.025,0.00250)
    };
    
    \addplot[color=black, dashed, line width=1pt] coordinates{
      (0.2,0.072)
      (0.1,0.018)
      (0.05,0.00450)
      (0.025,0.00110)
    };
    
    \legend{$\widetilde{A}_\eps^{\text{ME}}$, $\overline{A}_\eps^{\text{ME}}$, $y \propto \eps$, $y \propto \eps^2$};
    \end{axis}
  \end{tikzpicture}
  \caption{Error~\eqref{eq:errorAB} between the homogenized coefficient $A_\star$ and the effective coefficients $\overline{A}_\eps^{\text{ME}}$ and $\widetilde{A}_\eps^{\text{ME}}$, as a function of $\eps$.}
  \label{fig:absoluteerrorAbar_cohen}
\end{figure}
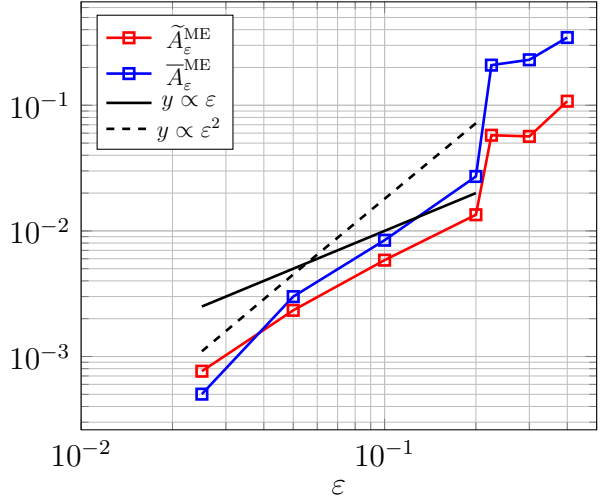

\begin{figure}
  \centering
  \begin{tikzpicture}[scale=1]
    \begin{axis}[
        xlabel=$\eps$,
        legend pos=north west,
        grid=both,
        xmin=0.001, xmax=0.41,
      ]        
      \addplot[color=green, mark=square, line width=1pt] coordinates{
        (0.4,0.506197428387412)
        (0.3,0.31561411914825654)
        (0.275,0.2935850458478148)
        (0.225,0.24828552577733126)
        (0.2,0.2605886663215005)
        (0.175,0.20274077516587577)
        (0.15,0.17581661190570572)
        (0.125,0.1700830909963723)
        (0.1,0.1358242935514575)
        (0.05,0.06554926904942866)
        (0.025, 0.03089304389077087) 
      };
      \addplot[color=red, mark=square, line width=1pt] coordinates{
        (0.4,0.370219027706)
        (0.3,0.2669395031)
        (0.275,0.237120492812)
        (0.225,0.198317945141)
        (0.2,0.253383770486)
        (0.175,0.182964188982)
        (0.15,0.151084241961)
        (0.125,0.165678532659)
        (0.1,0.132182283503)
        (0.05,0.0642822477006)
      };
      \addplot[color=blue, mark=square, line width=1pt] coordinates{
        (0.4,0.30858959564688465)
        (0.3,0.24358776885836592)
        (0.275,0.20405252168941213)
        (0.225,0.20251251182164542)
        (0.2,0.24795174182750587)
        (0.175,0.1513967014973608)
        (0.15,0.12627058861298684)
        (0.125,0.16247646733830642)
        (0.1,0.13095018539702805)
        (0.05,0.06402922015084547)
        (0.025, 0.031050925879217033) 
      };
      \legend{$A_\star$, $\widetilde{A}_\eps^{\text{ME}}$, $\overline{A}_\eps^{\text{ME}}$};
    \end{axis}
  \end{tikzpicture}
  \caption{Error~\eqref{eq:criteriarelativ} (computed with $Q=11$) for $\overline{A} \in \left\{A_\star, \overline{A}_\eps^{\text{ME}}, \widetilde{A}_\eps^{\text{ME}} \right\}$, as a function of $\eps$.}
  \label{fig:absoluteerror_cohen}
\end{figure}

\subsubsection{Random stationary case} \label{sec:random_case}

We now focus on the random setting presented in Section~\ref{subsection:testcaserandom}. 

Figure~\ref{fig:abarastarrandom} shows the relative error~\eqref{eq:errorAB} between the homogenized coefficient $A_\star$ and any of the effective coefficients $\overline{A}^{\text{ME}}_\eps$, $\overline{A}^{\text{MS}}_\eps$ and $\overline{A}^{\text{MV}}_\eps$. We recall that the value found for $\overline{A}^{\text{ME}}_\eps$, $\overline{A}^{\text{MS}}_\eps$ and $\overline{A}^{\text{MV}}_\eps$ explicitly depends on the specific realizations of $a_{\text{rand}}(\cdot, \omega)$ chosen to compute the expectation involved in the objective function. The errors we plot are therefore random and we show their mean values supplemented with a 95\% confidence interval.

Figure~\ref{fig:abarastarrandom} shows that, as $\eps \to 0$, all the strategies allow to recover the homogenized coefficient $A_\star$.

We observe that the coefficients $\overline{A}_\eps^{\text{ME}}$ and $\overline{A}_\eps^{\text{MS}}$ coincide (their difference is of the order of 0.1\%). We recall (see Section~\ref{section:section4}) that this is expected from a theoretical perspective, when maximizing over all $g \in L^2_m(\partial \mathcal{D})$. In practice, we have observed some mismatch in the periodic example presented in Section~\ref{sec:periodic_case}, which is presumably due to the approximation of the supremum~\eqref{eq:supproblem} by a maximum over a finite dimensional space, and the fact that the boundary conditions selected in Section~\ref{section:changingsupbymax} to generate this finite dimensional space were identified based on the assumptions that the effective and homogenized diffusion coefficients at stake are scalar, an assumption which does not hold for the periodic example of Section~\ref{sec:periodic_case}. In the random setting considered here, this assumption is now satisfied (exactly for $A_\star$, which is equal to $8 \, \text{Id}$, and with an excellent accuracy for the effective coefficients $\overline{A}_\eps^{\text{ME}}$ and $\overline{A}_\eps^{\text{MS}}$, for which the two diagonal coefficients typically differ by 0.1\% and for which the off-diagonal coefficient is 1000 times smaller than the diagonal coefficients). We presume that this particularity implies that maximizing over $V_n^P(\partial\mathcal{D})$ rather than $L^2_m(\partial \mathcal{D})$ makes a smaller difference, hence the matching of the two coefficients $\overline{A}_\eps^{\text{ME}}$ and $\overline{A}_\eps^{\text{MS}}$.


\begin{figure}[htbp]
    \centering
    \begin{tikzpicture}[scale=1]
      \begin{axis}[
          xlabel=$\eps$,
          legend pos=north west,
          grid=both,
          legend style={nodes={scale=0.5}},
      ]        
          \addplot[color=red, mark=square, line width=1pt] coordinates{
            (0.25,  0.0994095386183342)
            (0.16666666,  0.0755055537385)
            (0.1,    0.0464879287853)
            (0.05,   0.023042349563819344)
            (0.025,  0.011422895089157583)
          };
          \addplot[color=red, mark=square, line width=1pt, dashed] coordinates{
            (0.25,  0.0924772519673775)
            (0.16666666,  0.0700037422248)
            (0.1,    0.0436974208034)
            (0.05,   0.021167325787407735)
            (0.025,  0.01034933927580919)
          };
          \addplot[color=red, mark=square, line width=1pt, dashed] coordinates{
            (0.25,  0.1063418252692909)
            (0.16666666,  0.0810073652522)
            (0.1,    0.0492784367672)
            (0.05,   0.024917373340230952)
            (0.025,  0.012496450902505975)
          };

          \addplot[color=orange, mark=square, line width=1pt] coordinates{
            (0.25,  0.09948988828327754)
            (0.16666666,  0.0756080783967)
            (0.1,    0.0465009741217)
            (0.05,   0.02323398548734707)
            (0.025,  0.011194693496057725)
          };
          \addplot[color=orange, mark=square, line width=1pt, dashed] coordinates{
            (0.25,  0.09255416225476085)
            (0.16666666,  0.070134187016)
            (0.1,    0.0437169258916)
            (0.05,   0.021213750040200523)
            (0.025,  0.010261705420688782)
          };
          \addplot[color=orange, mark=square, line width=1pt, dashed] coordinates{
            (0.25,  0.10642561431179423)
            (0.16666666,  0.0810819697774)
            (0.1,    0.0492850223518)
            (0.05,   0.02525422093449362)
            (0.025,  0.012127681571426668)
          };

          \addplot[color=brown, mark=square, line width=1pt] coordinates{
            (0.25,  0.0810006949548212)
            (0.16666666,  0.0556948946575)
            (0.1,    0.0325823486114)
            (0.05,   0.014123217290285265)
            (0.025,  0.0061492704951020384)
          };
          \addplot[color=brown, mark=square, line width=1pt, dashed] coordinates{
            (0.25,  0.07425774971169402)
            (0.16666666,  0.0502905574333)
            (0.1,    0.0292613113113)
            (0.05,   0.0122143222451513)
            (0.025,  0.005265053521258134)
          };
          \addplot[color=brown, mark=square, line width=1pt, dashed] coordinates{
            (0.25,  0.08774364019794839)
            (0.16666666,  0.0610992318818)
            (0.1,    0.0359033859116)
            (0.05,   0.016032112335419228)
            (0.025,  0.007033487468945943)
          };

          \legend{$\overline{A}_\eps^{\text{ME}}$, CI $95\%$, ,$\overline{A}_\eps^{\text{MS}}$, CI $95\%$, ,$\overline{A}_\eps^{\text{MV}}$, CI $95\%$,};
      \end{axis}
  \end{tikzpicture}
\caption{Error~\eqref{eq:errorAB} between the homogenized coefficient $A_\star$ and the effective coefficients $\overline{A}_\eps^{\text{ME}}$, $\overline{A}_\eps^{\text{MS}}$ and $\overline{A}_\eps^{\text{MV}}$, as a function of $\eps$. At the scale of the figure, the curves corresponding to the effective coefficients $\overline{A}_\eps^{\text{ME}}$ and $\overline{A}_\eps^{\text{MS}}$ lie on top of each other.}
\label{fig:abarastarrandom}
\end{figure}
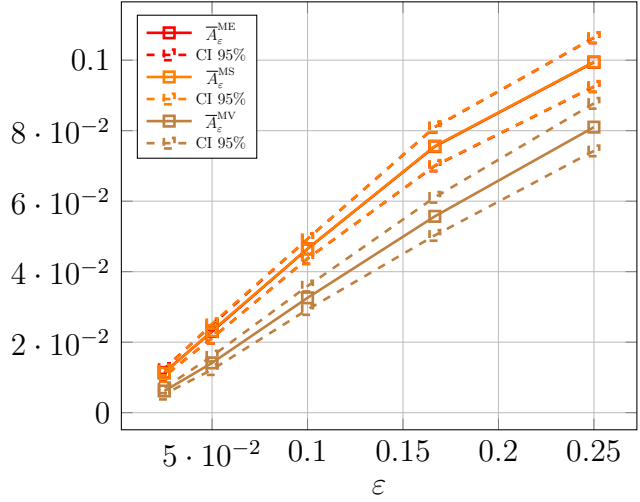

\medskip

Similarly as in the periodic case, we define the criterion
\begin{equation} \label{eq:criteriarelativrandom}
  \text{Err}^{\mathbb{E}}_{\eps, Q}(\overline{A}) = \sup_{g \in V_n^Q(\partial\mathcal{D})} \frac{\left\| \mathbb{E}\left[u_\eps(g)\right] - u(\overline{A}, g) \right\|_{L^2(\mathcal{D})}}{\left\| \mathbb{E}\left[u_\eps(g)\right] \right\|_{L^2(\mathcal{D})}}
\end{equation}
in order to assess the quality of the different effective operators in comparison to the oscillatory operator. 
In practice, we choose $Q = 11 \gg P$, so that we test the effective coefficients $\overline{A}_\eps$ (computed with $P$ boundary conditions) on a large number of new boundary conditions. Figure~\ref{fig:ubarustarrandom} shows the value of $\text{Err}^{\mathbb{E}}_{\eps, Q}$ for the different effective coefficients. 

For small values of $\eps$, all the coefficients tend to perform equally accurately, and with the accuracy of the homogenized coefficient $A_\star$. For instance, $\text{Err}^{\mathbb{E}}_{\eps, Q}(\overline{A})$ is of the order of 4\% for $\eps=0.025$, regardless of the choice of $\overline{A} \in \left\{ A_\star, \overline{A}_\eps^{\text{ME}}, \overline{A}_\eps^{\text{MS}}, \overline{A}_\eps^{\text{MV}} \right\}$.

For larger values of $\eps$, we observe that the homogenized coefficient $A_\star$ yields an approximation error that is larger (by a few percent) than the error associated with the effective coefficients obtained through the other approaches (i.e. an error of the order of 25\% instead of 20\% for $\eps=0.25$). The three coefficients $\overline{A}_\eps^{\text{ME}}$, $\overline{A}_\eps^{\text{MS}}$ and $\overline{A}_\eps^{\text{MV}}$ provide comparable errors (the 95\% confidence intervals overlap). 


\begin{figure}[htbp]
    \centering
    \begin{tikzpicture}[scale=1]
      \begin{axis}[
          xlabel=$\eps$,
          legend pos=north west,
          grid=both,
          legend style={nodes={scale=0.6}},
      ]     
        \addplot[color=blue, mark=square, line width=1pt] coordinates{
            (0.25, 0.2514602049856317)
            (0.16666666, 0.218833656532)
            (0.1,   0.14454040931)
            (0.05,  0.08299933399243875)
            (0.025, 0.042857683393596094)
      };
      \addplot[color=blue, mark=square, line width=1pt, dashed] coordinates{
            (0.25, 0.23935003493172707)
            (0.16666666, 0.207693063889)
            (0.1,   0.138091475711)
            (0.05,  0.07924721988123462)
            (0.025, 0.04104579117853025)
      };
      \addplot[color=blue, mark=square, line width=1pt, dashed] coordinates{
            (0.25, 0.26357037503953634)
            (0.16666666, 0.229974249175)
            (0.1,   0.15098934291)
            (0.05,  0.08675144810364288)
            (0.025,0.044669575608661936)
      };

        \addplot[color=red, mark=square, line width=1pt] coordinates{
            (0.25, 0.1886136098426988)
            (0.16666666, 0.17211115945)
            (0.1,   0.112935438125)
            (0.05,  0.06857062376649105)
            (0.025, 0.0358088173934682)%
      };
      \addplot[color=red, mark=square, line width=1pt, dashed] coordinates{
            (0.25, 0.1771213873016212)
            (0.16666666, 0.162238310842)
            (0.1,   0.106698072793)
            (0.05,  0.06527463556538053)
            (0.025, 0.03434658791387013)
      };
      \addplot[color=red, mark=square, line width=1pt, dashed] coordinates{
            (0.25, 0.20010583238377638)
            (0.16666666, 0.181984008058)
            (0.1,   0.119172803457)
            (0.05,  0.07186661196760156)
            (0.025, 0.03740310439970034)
      };

        \addplot[color=orange, mark=square, line width=1pt] coordinates{
            (0.25, 0.1886136098426988)
            (0.16666666, 0.172123227265)
            (0.1,   0.112925361667)
            (0.05,  0.06949868613481484)
            (0.025, 0.03575711988898486)
      };
      \addplot[color=orange, mark=square, line width=1pt, dashed] coordinates{
            (0.25, 0.1771213873016212)
            (0.16666666, 0.162246467687)
            (0.1,   0.106679646221)
            (0.05,  0.06638212328980393)
            (0.025, 0.03411113537826938)
      };
      \addplot[color=orange, mark=square, line width=1pt, dashed] coordinates{
            (0.25, 0.20010583238377638)
            (0.16666666, 0.181999986842)
            (0.1,   0.119171077113)
            (0.05,  0.07261524897982576)
            (0.025, 0.03740310439970034)
      };

        \addplot[color=brown, mark=square, line width=1pt] coordinates{
            (0.25, 0.19696968408166557)
            (0.16666666, 0.18169927702)
            (0.1,   0.121269318634)
            (0.05,  0.07299271246850803)
            (0.025, 0.038687219908412944)
      };
      \addplot[color=brown, mark=square, line width=1pt, dashed] coordinates{
            (0.25, 0.1847774825547361)
            (0.16666666, 0.170926078667)
            (0.1,   0.114762346464)
            (0.05,  0.0693644647654533)
            (0.025, 0.03693270366861936)
      };
      \addplot[color=brown, mark=square, line width=1pt, dashed] coordinates{
            (0.25, 0.20916188560859503)
            (0.16666666, 0.192472475372)
            (0.1,   0.127776290805)
            (0.05,  0.07662096017156275)
            (0.025, 0.04044173614820653)
      };

          \legend{$A_\star$, CI $95\%$, ,$\overline{A}_\eps^{\text{ME}}$, CI $95\%$, ,$\overline{A}_\eps^{\text{MS}}$, CI $95\%$, ,$\overline{A}_\eps^{\text{MV}}$, CI $95\%$, };
      \end{axis}
  \end{tikzpicture}
\caption{Error~\eqref{eq:criteriarelativrandom} (computed with $Q=11$) for $\overline{A} \in \left\{ A_\star, \overline{A}_\eps^{\text{ME}}, \overline{A}_\eps^{\text{MS}}, \overline{A}_\eps^{\text{MV}} \right\}$, as a function of $\eps$.}
\label{fig:ubarustarrandom}
\end{figure}
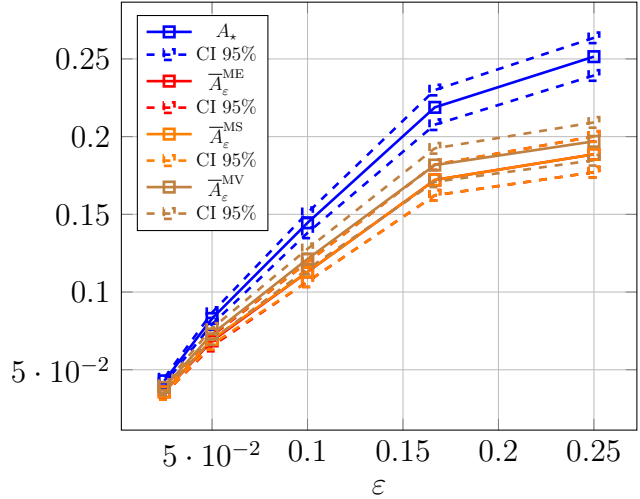

\subsubsection{Robustness with respect to noise} \label{sec:noise_general}

In this section, we explore the robustness of our approach with respect to noise. Considering the deterministic periodic test case introduced in Section~\ref{subsection;testcaseperiodic}, we investigate two approaches. In the first one, noise is regarded as an undesired discrepancy in the measurement. In the second one, noise is purposely introduced in the optimization procedure to enhance robustness of the output. 

\paragraph{Noise as a disturbance.} We explore here the robustness of our approach with respect to the presence of noise in the value of the energy. To fix the ideas, we choose $\eps = 0.025$ and select the first $P=3$ boundary conditions $\phi_1$, $\phi_2$ and $\phi_3$. We introduce a multiplicative noise, so that the relation between the noisy energy $\mathcal{E}(A_\eps, \phi_i;\sigma)$ and the exact energy $\mathcal{E}(A_\eps, \phi_i)$ is modeled by
\begin{equation} \label{eq:noisemodel}
  \mathcal{E}(A_\eps, \phi_i; \sigma) = (1 + \sigma \, \eta_i) \, \mathcal{E}(A_\eps, \phi_i),
\end{equation}
where $\{ \eta_i \}_{1 \leq i \leq 3}$ denote i.i.d. random Gaussian variables with mean $0$ and variance $1$, and where $\sigma$ quantifies the magnitude of the noise. We denote by $\overline{A}_{\eps, \sigma}^{\text{ME}}$ the constant effective coefficient computed on the basis of these noisy values of energy. Notice that the noise model~\eqref{eq:noisemodel} can be interpreted as an approximate model accounting for an error $\delta\phi_i \in L^2_m(\partial \mathcal{D})$ associated to the boundary condition $\phi_i$. Indeed, we have
\begin{align*}
  \mathcal{E}(A_\eps, \phi_i +\delta\phi_i)
  &= - \frac{1}{2} \int_{\partial \mathcal{D}} (\phi_i +\delta\phi_i) \, u_\eps(\phi_i +\delta\phi_i)
  \\
  &= - \frac{1}{2} \int_{\partial \mathcal{D}} \phi_i \, u_\eps(\phi_i) - \frac{1}{2} \int_{\partial \mathcal{D}} \delta\phi_i \, u_\eps(\phi_i) - \frac{1}{2} \int_{\partial \mathcal{D}} \phi_i \, u_\eps(\delta \phi_i) - \frac{1}{2} \int_{\partial \mathcal{D}} \delta\phi_i \, u_\eps(\delta\phi_i)
  \\
  &= - \frac{1}{2} \int_{\partial \mathcal{D}} \phi_i \, u_\eps(\phi_i) - \int_{\partial \mathcal{D}} \delta\phi_i \, u_\eps(\phi_i) + o(\|\delta\phi_i\|_{L^2(\partial\mathcal{D})})
  \\
  &= \left( 1 + \underbrace{\frac{2 \int_{\partial \mathcal{D}} \delta\phi_i \, u_\eps(\phi_i)}{\int_{\partial \mathcal{D}} \phi_i \, u_\eps(\phi_i) }}_{\text{approx. by the random number $\sigma \, \eta_i$}} \right)\mathcal{E}(A_\eps, \phi_i) + o(\|\delta\phi_i\|_{L^2(\partial\mathcal{D})}),
\end{align*}
where we have used that $\dis \int_{\partial \mathcal{D}} \phi_i \, u_\eps(\delta\phi_i) = \int_{\partial \mathcal{D}} \delta\phi_i \, u_\eps(\phi_i)$.

Figure~\ref{fig:errorgaussiannoise} shows the mean value and 95\% confidence intervals (evaluated on the basis of 40 independent realizations of $\{ \eta_i \}_{1 \leq i \leq 3}$, that lead to 40 independent realizations of $\overline{A}_{\eps, \sigma}^{\text{ME}}$) of the relative $L^2$-error between the coefficient $\overline{A}_\eps^{\text{ME}}$ (computed using the exact energy values) and $\overline{A}_{\eps, \sigma}^{\text{ME}}$ (the $L^2$ norm of a matrix $B$ being defined by $\dis \| B \|_2 = \sqrt{\sum_{1\leq i \leq j \leq d} B_{ij}^2}$). The results indicate that the error grows linearly with the noise standard deviation $\sigma$. The proportionality factor is around $4$ (for instance, a 2\% error in the energy measurements results in a 9\% error in the reconstructed coefficient).

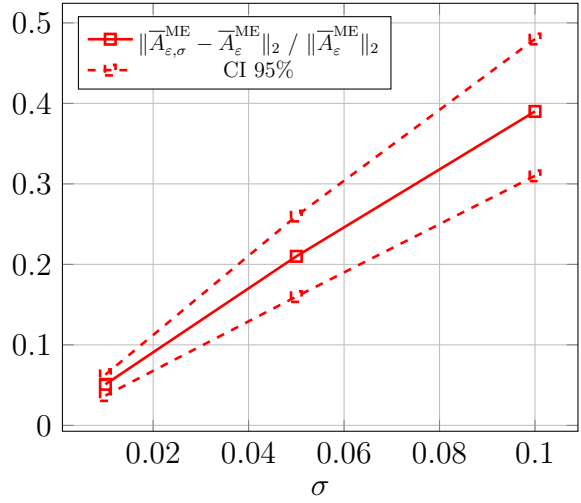
\begin{figure}[htbp]
    \centering
    \begin{tikzpicture}[scale=1]
      \begin{axis}[
          xlabel=$\sigma$,
          legend pos=north west,
          grid=both,
          every x tick label/.append style={/pgf/number format/.cd,fixed,precision=2},
          legend style={nodes={scale=0.7}}
      ]  
      \addplot[color=red, mark=square, line width=1pt] coordinates{
        (0.1,0.39)
        (0.05,0.21)
        (0.01,0.051)
      };
      \addplot[color=red, mark=square, line width=1pt, dashed] coordinates{
        (0.1,0.31)
        (0.05,0.16)
        (0.01,0.037)
      };
      \addplot[color=red, mark=square, line width=1pt, dashed] coordinates{
        (0.1,0.48)
        (0.05,0.26)
        (0.01,0.063)
      };

      \legend{$\|\overline{A}_{\eps, \sigma}^{\text{ME}} - \overline{A}_\eps^{\text{ME}}\|_2 \ / \ \|\overline{A}_\eps^{\text{ME}}\|_2$, CI $95\%$};
      \end{axis}
    \end{tikzpicture}
    \caption{Error $\frac{\|\overline{A}_{\eps, \sigma}^{\text{ME}} - \overline{A}_\eps^{\text{ME}}\|_2}{\|\overline{A}_\eps^{\text{ME}}\|_2}$ as a function of the noise standard deviation $\sigma$ (for $\eps=0.025$).}
    \label{fig:errorgaussiannoise}
\end{figure}

\paragraph{Introducing noise on purpose.} We now consider an alternative viewpoint. Inspired by ideas from Recurrent Neural Networks (see e.g.~\cite[Section 7.6.2]{bengio2017deep}), we treat noise as a constructive element. More precisely, we investigate how noise can be purposely added in the modeling procedure to enhance the robustness of the output effective coefficient $\overline{A}$. 

\subparagraph{Modeling.} Mathematically, the problem is formulated as
\begin{equation} \label{eq:infsupnoiseApipeline}
  \inf_{\overline{A} \in \mathcal{S}_{\alpha, \beta}} \ \sup_{\scriptsize \begin{aligned} &g \in L^2_m(\partial \mathcal{D}), \\[-0.5em] \|&g\|_{L^2(\partial \mathcal{D})}=1 \end{aligned}} \Big| \mathcal{E}(A_\eps, g) - \mathbb{E}\left(\mathcal{E}(\overline{A} + \eta, g)\right) \Big|^2,
\end{equation}
where $\eta$ denotes a random (matrix-valued) variable. 

In this framework, noise is introduced at the level of the coefficient $\overline{A}$. A typical motivation for this stems from issues of experimental reproducibility. To illustrate this fact, suppose that Problem~\eqref{eq:infsup} admits an optimal solution $\overline{A}^{\text{opt}}$. The material represented by $\overline{A}^{\text{opt}}$ provides the best effective material describing the heterogeneous problem~\eqref{eq:diffusive}. However, in practice, the exact manufacturing of $\overline{A}^{\text{opt}}$ may be unfeasible. Instead, one may only be able to construct an approximation $\widetilde{A}^{\text{opt}}$, and there is of course in general no guarantee that $\widetilde{A}^{\text{opt}}$ captures the properties of the heterogeneous problem~\eqref{eq:diffusive}. By inserting noise into the parameter $\overline{A}$ through the random variable $\eta$, we enhance the robustness with respect to perturbations of the optimized effective coefficient. 

\subparagraph{An example in dimension one.} In order to intuitively understand how introducing noise in this manner affects the effective coefficient, we temporarily consider the one-dimensional setting. Despite its extreme simplicity, the setting provides a useful guideline for higher dimensional settings. We fix $\mathcal{D}= (0,1)$. First, we observe that the unit sphere of $L^2_m(\partial \mathcal{D})$ restricts (up to a sign change) to the single boundary condition $g$ defined by $g(0) = -g(1) = \sqrt{2}^{-1}$. Second, we make some simplifications by considering the asymptotic regime $\eps \rightarrow 0$ and by simplifying the constraints on $\overline{A}$. Problem~\eqref{eq:infsupnoiseApipeline} becomes
\begin{equation} \label{eq:infsupnoiseApipeline1D}
  \inf_{\overline{a}>0} \ \Big| \mathcal{E}(a_\star, g) - \mathbb{E}\left(\mathcal{E}(\overline{a} + \eta, g)\right) \Big|^2.
\end{equation}
Then, we observe that the solution $u(a_\star,g)$ is given by $w(g)/a_\star$, with $w(g)$ the unique zero-mean solution to $\dis -\frac{d^2 w(g)}{dx^2} = 0$ in $(0,1)$ with the boundary conditions $\dis \frac{d w(g)}{dx}(x=0)=-g(0)$ and $\dis \frac{d w(g)}{dx}(x=1)=g(1)$. Consequently, the energy $\mathcal{E}(a_\star, g)$ writes $\dis \frac{-1}{2 \, a_\star} \, \int_{\partial \mathcal{D}} g \, w(g)$. A similar expression is found for $\mathcal{E}(\overline{a}+\eta, g)$. Hence, up to an irrelevant multiplicative constant, Problem~\eqref{eq:infsupnoiseApipeline1D} writes
\begin{equation*}
  \inf_{\overline{a}>0} \ \left| \frac{1}{a_\star} - \mathbb{E}\left(\frac{1}{\overline{a} + \eta}\right) \right|^2.
\end{equation*}
In the case where $\eta$ follows a uniform distribution in $[\alpha_1,\alpha_2]$ (with $\alpha_2 > \alpha_1 \geq 0$), we compute
\begin{equation*}
  f(\overline{a}) := \left| \frac{1}{a_\star} - \mathbb{E}\left(\frac{1}{\overline{a} + \eta}\right) \right|^2 = \left| \frac{1}{a_\star} - \frac{1}{\alpha_2-\alpha_1} \, \ln\left(\frac{\overline{a}+\alpha_2}{\overline{a}+\alpha_1} \right) \right|^2.
\end{equation*}
The minimizer can be identified analytically. Introducing the notation $\ell = \alpha_2-\alpha_1$, its expression depends on the sign of $\tau := \ln(\alpha_2/\alpha_1) - \ell/a_\star$. We only consider here the case when $\tau > 0$, which for instance holds whenever $\alpha_2 < a_\star$ (using that $y \ln y \geq y -1$ for any positive $y$, we systematically have that $\ln(\alpha_2/\alpha_1) \geq \ell/\alpha_2$, and the fact that $\alpha_2 < a_\star$ then implies that $\tau > 0$). In the case when $\tau > 0$, the function $f$ attains its minimum at
\begin{equation*}
  \overline{a}^{\text{opt}} = \frac{\alpha_2 - \alpha_1 \exp(\ell/a_\star)}{\exp(\ell/a_\star) - 1} > 0.
\end{equation*}
We observe that $\overline{a}^{\text{opt}}$ does not coincide with $a_\star$, which would be the minimizer of $f$ in the case where no noise is added. In the case where $\alpha_1 = 0$ and $\alpha_2 \rightarrow 0$, the optimizer $\overline{a}^{\text{opt}}$ converges to $a_\star$.

\subparagraph{Numerical simulations in dimension two.} The algorithm presented in Section~\ref{subsection:algorithm} to solve~\eqref{eq:infsup} can be adapted mutatis mutandis to solve~\eqref{eq:infsupnoiseApipeline}. In particular, the expectation in~\eqref{eq:infsupnoiseApipeline} is approximated by a mean over $M_1=40$ independent realizations of $\eta$. The supremum over $L^2_m(\partial \mathcal{D})$ is again replaced by a supremum over the space spanned by the first modes $\{ \phi_p \}_{1 \leq p \leq P}$ of the operator $\mathcal{R}$ defined by~\eqref{eq:laplacianinverse}. 

\smallskip

In our experiment, we assume that the components of $\eta$ are independent from one another, and that each of them is distributed according to a Gaussian distribution with zero mean and variance $\sigma^2$. We denote the optimizer to~\eqref{eq:infsupnoiseApipeline} by $\overline{A}_{\eps, \sigma}^{\text{ME}}$. We fix $\eps = 0.05$ and $P=3$.

Figure~\ref{fig:errorcoeffnoiseonpipeline} shows the relative difference between the coefficients $\overline{A}_{\eps, \sigma}^{\text{ME}}$ and $\overline{A}_\eps^{\text{ME}}$ as a function of $\sigma$. We also provide $95\%$ confidence intervals, which are obtained by considering $M_2=40$ resolutions of~\eqref{eq:infsupnoiseApipeline} using independent realizations of $\eta$ to estimate the expectations. Even though the error increases when $\sigma$ increases (which is expected), it remains small even for relatively large values of $\sigma$. For instance, when $\sigma = 2$, the error between the coefficients $\overline{A}_{\eps, \sigma}^{\text{ME}}$ and $\overline{A}_\eps^{\text{ME}}$ is about $3\%$. This error is small compared to the perturbation applied to the coefficient $\overline{A}$. Indeed, $\sigma = 2$ corresponds to a perturbation 
of roughly $10\%$, when compared to the unperturbed reference matrix $\overline{A}_\eps^{\text{ME}}$ (its largest entry is $[\overline{A}_\eps^{\text{ME}}]_{11} \approx 20$). This suggests that the coefficient $\overline{A}_\eps^{\text{ME}}$ is robust to (significant) perturbations and that the identification procedure is actually stable. 

\begin{figure}[!ht]
  \centering
  \begin{tikzpicture}[scale=1]
    \begin{axis}[
        xlabel=$\sigma$,
        legend pos=north west,
        grid=both,
        every x tick label/.append style={/pgf/number format/.cd,fixed,precision=2},
        legend style={nodes={scale=0.7}}
      ]  
      \addplot[color=red, mark=square, line width=1pt] coordinates{
        (4,0.13)
        (2,0.033)
        (1,0.011)
        (0.1,0.0010)
      };
      \addplot[color=red, mark=square, line width=1pt, dashed] coordinates{
        (4,0.12)
        (2,0.031)
        (1,0.011)
        (0.1,0.0010)
      };
      \addplot[color=red, mark=square, line width=1pt, dashed] coordinates{
        (4,0.14)
        (2,0.035)
        (1,0.012)
        (0.1,0.0010)
      };
      
      \legend{$\|\overline{A}_{\eps, \sigma}^{\text{ME}}-\overline{A}_\eps^{\text{ME}}\|_2 \ / \ \|\overline{A}_\eps^{\text{ME}}\|_2$, CI $95\%$};
    \end{axis}
  \end{tikzpicture}
  \caption{Relative difference $\frac{\|\overline{A}_{\eps, \sigma}^{\text{ME}}-\overline{A}_\eps^{\text{ME}}\|_2}{\|\overline{A}_\eps^{\text{ME}}\|_2}$ as a function of the noise standard deviation $\sigma$ (for $\eps=0.05$).}
  \label{fig:errorcoeffnoiseonpipeline}
\end{figure}
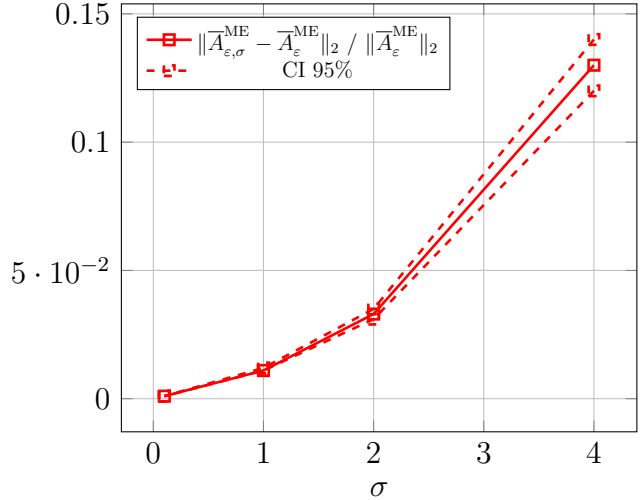

\section*{Acknowledgments}

The authors would like to thank Emmanuel Baranger, Ludovic Chamoin and Federica Daghia at ENS Paris-Saclay for stimulating and enlightening discussions, and for the detailed presentation of the experimental aspects of their research that helped designing the strategy. They are also grateful to Albert Cohen at Sorbonne Universit\'e for his many comments about this work, and to Habib Ammari at ETH Z\"urich for pointing out several interesting and practically relevant approaches for inverse problems. They thank Olivier Pantz at Universit\'e de Nice for his remarks on a preliminary version of this manuscript.

The first two authors are grateful to the ONR and the EOARD for their continuous support, currently under grants N00014-25-1-2299 and FA8655-24-1-7057, respectively.

\bibliographystyle{plain}
\bibliography{biblio}

\begin{thebibliography}{10}

\bibitem{abdulle2019numerical}
A.~Abdulle and A.~Di~Blasio.
\newblock Numerical homogenization and model order reduction for multiscale
  inverse problems.
\newblock {\em Multiscale Modeling \& Simulation}, 17(1):399--433, 2019.

\bibitem{zbMATH06685276}
H.~Ammari, T.~Boulier, J.~Garnier, and H.~Wang.
\newblock Shape recognition and classification in electro-sensing.
\newblock {\em Proc. Natl. Acad. Sci. USA}, 111(32):11652--11657, 2014.

\bibitem{zbMATH06566400}
H.~Ammari, J.~Garnier, L.~Giovangigli, W.~Jing, and J.-K. Seo.
\newblock Spectroscopic imaging of a dilute cell suspension.
\newblock {\em J. Math. Pures Appl.}, 105(5):603--661, 2016.

\bibitem{zbMATH02119069}
H.~Ammari and H.~Kang.
\newblock {\em Reconstruction of small inhomogeneities from boundary
  measurements}, volume 1846 of {\em Lect. Notes Math.}
\newblock Berlin: Springer, 2004.

\bibitem{ammari_uhlmann}
H.~Ammari and G.~Uhlmann.
\newblock Reconstruction of the potential from partial {C}auchy data for the
  {S}chr\"odinger equation.
\newblock {\em Indiana University Mathematics Journal}, 53(1):169--184, 2004.

\bibitem{papanicolau1978asymptotic}
A.~Bensoussan, J.-L. Lions, and G.~Papanicolaou.
\newblock {\em Asymptotic analysis for periodic structures}.
\newblock AMS Chelsea Publishing, 2011.
\newblock (reprint of the 1978 original with corrections and bibliographical
  additions).

\bibitem{bhattacharya2023learning}
K.~Bhattacharya, N.~Kovachki, A.~Rajan, A.~M. Stuart, and M.~Trautner.
\newblock Learning homogenization for elliptic operators.
\newblock {\em SIAM Journal on Numerical Analysis}, 62(4):1844--1873, 2024.

\bibitem{blanc2023homogenization}
X.~Blanc and C.~Le~Bris.
\newblock {\em Homogenization theory for multiscale problems: an introduction},
  volume~21 of {\em Modeling, Simulation and Applications}.
\newblock Springer, 2023.

\bibitem{halikias2023}
N.~Boull\'e, D.~Halikias, and A.~Townsend.
\newblock Elliptic {PDE} learning is provably data-efficient.
\newblock {\em Proceedings of the National Academy of Sciences},
  120(39):e2303904120, 2023.

\bibitem{cherkaev2001inverse}
E.~Cherkaev.
\newblock Inverse homogenization for evaluation of effective properties of a
  mixture.
\newblock {\em Inverse Problems}, 17:1203--1218, 2001.

\bibitem{cherkaev2008dehomogenization}
E.~Cherkaev and M.-J.~Y. Ou.
\newblock Dehomogenization: reconstruction of moments of the spectral measure
  of the composite.
\newblock {\em Inverse Problems}, 24(6):065008, 2008.

\bibitem{chung2023multi}
E.~Chung, W.~T. Leung, S.-M. Pun, and Z.~Zhang.
\newblock Multi-agent reinforcement learning aided sampling algorithms for a
  class of multiscale inverse problems.
\newblock {\em Journal of Scientific Computing}, 96(2):55, 2023.

\bibitem{frederick2014numerical}
C.~Frederick and B.~Engquist.
\newblock Numerical methods for multiscale inverse problems.
\newblock {\em Communications in Mathematical Sciences}, 7:305--328, 2017.

\bibitem{bengio2017deep}
I.~Goodfellow, Y.~Bengio, and A.~Courville.
\newblock {\em Deep learning}.
\newblock Adaptive Computation and Machine Learning series. MIT press, 2016.

\bibitem{hecht_freefem}
F.~Hecht.
\newblock New development in {F}ree{F}em++.
\newblock {\em J. Numer. Math.}, 20(3-4):251--266, 2012.

\bibitem{le2018best}
C.~Le~Bris, F.~Legoll, and S.~Lemaire.
\newblock On the best constant matrix approximating an oscillatory
  matrix-valued coefficient in divergence form operators.
\newblock {\em ESAIM: Control, Optimisation and Calculus of Variations},
  24(4):1345--1380, 2018.

\bibitem{le2013approximation}
C.~Le~Bris, F.~Legoll, and K.~Li.
\newblock Approximation grossi\`ere d'un probl\`eme elliptique \`a coefficients
  hautement oscillants [{C}oarse approximation of an elliptic problem with
  highly oscillatory coefficients].
\newblock {\em C. R. Acad. Sci. Paris, Ser. I}, 351(7-8):265--270, 2013.

\bibitem{lions2005some}
J.-L. Lions.
\newblock Some aspects of modeling problems in distributed parameter systems.
\newblock In A.~Ruberti, editor, {\em Distributed Parameter Systems: Modeling
  and Identification: Proceedings of the IFIP Working Conference Rome, Italy,
  June 21--24, 1976}, pages 11--41. Springer, 1978.

\bibitem{lochner2023identification}
T.~Lochner and M.~A. Peter.
\newblock Identification of microstructural information from macroscopic
  boundary measurements in steady-state linear elasticity.
\newblock {\em Mathematical Methods in the Applied Sciences}, 46(1):1295--1316,
  2023.

\bibitem{nolen2009fine}
J.~Nolen and G.~Papanicolaou.
\newblock Fine scale uncertainty in parameter estimation for elliptic
  equations.
\newblock {\em Inverse Problems}, 25:115021, 2009.

\bibitem{nolen2012multiscale}
J.~Nolen, G.~A. Pavliotis, and A.~M. Stuart.
\newblock Multiscale modelling and inverse problems.
\newblock In I.~G. Graham, T.~Hou, O.~Lakkis, and R.~Scheichl, editors, {\em
  Numerical Analysis of Multiscale Problems}, volume~83 of {\em Lecture Notes
  in Computationnal Science and Engineering}, pages 1--34. Springer, 2012.

\bibitem{papanicolaou1979boundary}
G.~Papanicolaou and S.~R.~S. Varadhan.
\newblock Boundary value problems with rapidly oscillating random coefficients.
\newblock In J.~Fritz, J.~L. Lebaritz, and D.~Szasz, editors, {\em Proc.
  Colloq. on Random Fields: Rigorous Results in Statistical Mechanics and
  Quantum Field Theory (Esztergom 1979), Colloq. Math. Soc., Janos Bolyai, 10,
  North-Holland, Amsterdam (1981)}, pages 853--873, 1981.

\bibitem{park2022physics}
J.~S.~R. Park and X.~Zhu.
\newblock Physics-informed neural networks for learning the homogenized
  coefficients of multiscale elliptic equations.
\newblock {\em Journal of Computational Physics}, 467:111420, 2022.

\bibitem{PhDthesis}
S.~Ruget.
\newblock {\em Effective approximations for multiscale {PDEs} based on limited
  information}.
\newblock PhD thesis, Ecole Nationale des Ponts et Chauss\'ees, 2025.
\newblock (available at {\tt https://hal.science/tel-05450477}).

\bibitem{stuart2010inverse}
A.~M. Stuart.
\newblock Inverse problems: a {B}ayesian perspective.
\newblock {\em Acta Numerica}, 19:451--559, 2010.

\end{thebibliography}

\end{document}